\documentclass[12pt,a4paper]{article}
\usepackage{template}
\usepackage{hyperref}
\usepackage{mathtools}
\usepackage{adjustbox}
\usepackage{tikz}
\usepackage[left=2.8cm, right=2.8cm]{geometry}
\usetikzlibrary{
	arrows,
	chains,
	matrix,
	positioning,
	scopes,
	decorations,
	decorations.markings,
	decorations.pathmorphing,
	decorations.pathreplacing
}
\usepackage{tikz-cd}
\usepackage[
	color=white,
	linecolor=gray,
	textwidth=1.8cm,
	disable,
]{todonotes}

\def\calA{{\mathcal{A}}}

\def\calB{{\mathcal{B}}}

\def\calC{{\mathcal{C}}}

\def\calH{{\mathcal{H}}}

\def\bbN{{\mathbb{N}}}
\def\calO{{\mathcal{O}}}

\def\calT{{\mathcal{T}}}

\def\calU{{\mathcal{U}}}

\def\calV{{\mathcal{V}}}

\def\calW{{\mathcal{W}}}

\def\bbZ{{\mathbb{Z}}}

\def\id{{\operatorname{id}}}

\def\im{{\operatorname{im}}}

\def\Hom{{\operatorname{Hom}}}
\def\Ob{{\operatorname{Ob}}}

\def\pr{{\operatorname{pr}}}
\def\Pr{{\operatorname{Pr}}}

\def\bary{{\operatorname{bary}}}

\def\girth{{\operatorname{girth}}}
\def\diam{{\operatorname{diam}}}
\def\inc{{\operatorname{incl}}}

\def\trans{{\operatorname{trans}}}
\def\sign{{\operatorname{sign}}}

\DeclareMathOperator*{\colim}{colim}
\DeclareMathOperator*{\argmin}{argmin}
\def\contr{\prod^\ctr}

\def\Ab{{\calA\textsl{b}}}

\def\Ch{\calC\!\textsl{h}\,}
\def\Chhfd{{\widetilde \Ch_{\!\operatorname{hfd}}}}
\def\chhfd{{\Ch_{\!\operatorname{hfd}}}}
\def\Chf{{\Ch_{\!\operatorname{f}}}}
\def\Chgeq{{\Ch^{\!\geq}}}
\def\Idem{{\operatorname{Idem}}}

\def\where{{\,\big|\,}}           
\def\comp{{\,\circ\,}}            
\def\pt{\mathsf{pt}}              
\def\ctr{\text{contr.}}           
\def\bd{\text{bound.}}            
\def\sing{{\operatorname{sing}}}  

\makeatletter
\def\underbrace#1{%
   \@ifnextchar_{\tikz@@underbrace{#1}}{\tikz@@underbrace{#1}_{}}}
\def\tikz@@underbrace#1_#2{%
   \tikz[baseline=(a.base)] {\node[inner sep=1] (a) {\(#1\)};
   \draw[line cap=round, line width = 0.75pt,decorate,decoration={brace,amplitude=5pt}]
     (a.south east) -- node[below,inner sep=4pt] {\(\scriptstyle #2\)} (a.south west);}}

\makeatletter
\def\overbrace#1{%
   \@ifnextchar_{\tikz@@overbrace{#1}}{\tikz@@overbrace{#1}_{}}}
\def\tikz@@overbrace#1_#2{%
   \tikz[baseline=(a.base)] {\node[inner sep=1] (a) {\(#1\)};
   \draw[line cap=round, line width = 0.75pt,decorate,decoration={brace,amplitude=5pt}]
     (a.north west) -- node[above,inner sep=4pt] {\(\scriptstyle #2\)} (a.north east);}}

\newcommand{\set}[2]{\left\{ #1 \,\middle|\, #2 \right\}}

\newcommand{\bigslant}[2]{{\left.\raisebox{.2em}{$#1$}\middle/\raisebox{-.2em}{$#2$}\right.}}

\makeatletter
\newcommand{\oset}[3][0ex]{%
  \mathrel{\mathop{#3}\limits^{
    \vbox to#1{\kern-4\ex@
    \hbox{$\scriptstyle#2$}\vss}}}}
\makeatother

\makeatletter
\newcommand{\osett}[3][0ex]{%
	\mathrel{\mathop{#3}\limits^{
			\vbox to#1{\kern-6\ex@
				\hbox{$\scriptstyle#2$}\vss}}}}
\makeatother

\makeatletter
\newcommand{\uset}[3][0ex]{%
  \mathrel{\mathop{#3}\limits^{
    \vbox to#1{\kern8\ex@
    \hbox{$\scriptstyle#2$}\vss}}}}
\makeatother

\def\lf{\textit{lf}}

\author{Markus Zeggel}
\title{The Bounded Isomorphism Conjecture\\ for Spaces of Graphs with Large Girth}

\begin{document}
	
	\maketitle
	
	\paragraph*{Abstract}
	In this article we study a coarse version of the $K$-theoretic Farrell--Jones conjecture we call \emph{coarse} or \emph{bounded isomorphism conjecture}.
	With techniques that have already been used to prove the Farrell--Jones conjecture for hyperbolic groups we are able to verify the bounded isomorphism conjecture for spaces of graphs with large girth and bounded geometry.

	\tableofcontents
	
	\section{Introduction}

There are two important isomorphism conjectures in the fields of operator algebras and geometric topology: the Baum--Connes conjecture and the Farrell--Jones conjecture. In the Davis--Lück picture \cite{Davis.1998}, these conjectures predict the following: For any group $G$, the assembly maps 
$$H^G_*(E_\text{Fin} G; \mathbb K^\text{top}) \rightarrow K_*(C^*_r(G))
\quad \text{ and } \quad
H^G_*(E_\text{VCyc} G; \mathbb K_R^\text{alg}) \rightarrow K_*(R[G])$$
are isomorphisms. These formulae help compute the topological $K$-theory of the reduced group $C^*\!$-algebra $C^*_r(G)$ and the algebraic $K$-theory of the group ring $R[G]$, respectively.

In order to prove or disprove such conjectures, it can be helpful to compare them to other versions of these conjectures.
For example, there is a coarse version of the Baum-Connes conjecture.
Using the connection between the ``coarse'' and the ``usual'' conjecture, one can show that the Baum--Connes assembly map fails to be surjective for Gromov monster groups, i.e.\ groups that coarsely contain infinite expanders, cf.\ \cite[Theorem~8.2]{Willett.2012a}.
Such methods of constructing counter-examples are based on the work of Higson, Lafforgue and Skandalis, see \cite{Higson.2002}.

In this article we study the same geometric objects as in \cite{Willett.2012a} and \cite{Willett.2012b}, namely space of graphs with large girth.

\begin{definition}
	\label{def:girth}
	Let $\Gamma = (V,E)$ be a graph. The \emph{girth} of $\Gamma$ is the minimal length of all cycles in $\Gamma$, i.e.\
	$$\girth(\Gamma) := \inf \set{ n \in \bbN}{\exists v_0, \dots, v_n \in V: v_0 = v_n \wedge (v_i, v_{i+1}) \in E, 0 \leq i < n}.$$
\end{definition}

\begin{definition}
	\label{def:space_of_graphs}
	Given a sequence of finite graphs $\Gamma_n$, we say $\Gamma = \bigsqcup_{n \in \bbN} \Gamma_n$ is a \emph{space of graphs with large girth} if $\girth(\Gamma_n) \uset{n \rightarrow \infty}{\longrightarrow} \infty$.
\end{definition}

\subsection{Result}

The reduced and the maximal version of the Baum--Connes conjecture behave very differently.
The reduced coarse assembly map fails to be surjective for spaces of expander graphs with bounded geometry and large girth, cf.\ \cite[Theorem 1.5 and Theorem 1.6]{Willett.2012a}.
In contrast, the maximal assembly map is an isomorphism for those spaces, whether they are expanders or not, cf.\ \cite[Theorem 1.1]{Willett.2012b}.
In this article we mimic the proof of the latter statement in the Farrell--Jones setting and eventually prove the following result.

\begin{theorem}
	\label{thm:main_theorem}
	Let $X$ be a space of graphs with large girth and bounded geometry. 
	The coarse assembly maps combine to an isomorphism
	$$\colim_{d \geq 0} H_*^\lf\big(P_d(X); \mathbb K_\calA^\text{alg}\big) \rightarrow K_*\big(\calC^\lf( X; \calA)\big).$$
\end{theorem}

\subsection{Outline}

After introducing some basic notions and recapitulating the bounded isomorphism conjecture in Section \ref{sec:basics} we will explain how to simplify the statements we have to prove in order to prove the conjecture in Section \ref{sec:getting_rid_of_the_rips_complex}.
In Section \ref{sec:proof_outline} we will introduce the commutative diagram which will explain why a certain obstruction vanishes. The actual proof of this makes up the main part of this article and will encompass the last four sections.

\subsubsection*{Acknowledgements}

This article is based on my PhD thesis. I would like to thank my supervisor Arthur Bartels for his continuing support and many helpful discussions.

The project was funded by the Deutsche Forschungsgemeinschaft (DFG, German Research Foundation) -- Project-ID 427320536 -- SFB 1442, as well as under Germany’s Excellence Strategy EXC 2044 390685587, Mathematics Münster: Dynamics--Geometry--Structure.
	\section{Basics}
\label{sec:basics}

\subsection{$\varepsilon$-Filtered Categories}

With the following definition we are able to apply concepts of basic analysis to category theory.

\begin{definition}
	We call an additive category $\calA$ an \emph{$\varepsilon$-filtered category} if, for every $\varepsilon > 0$ and pair of objects $X,Y \in \calA$, there is a collection of (additive) subgroups $\Hom_\varepsilon(X, Y) \subseteq \Hom(X,Y)$, called \emph{the subgroup of $\varepsilon$-controlled morphisms}, such that
	\begin{enumerate}
		\item $\varepsilon \leq \varepsilon' \Rightarrow \Hom_\varepsilon(X,Y) \subseteq \Hom_{\varepsilon'}(X,Y)$,
		\item $\bigcup_{\varepsilon > 0} \Hom_\varepsilon(X,Y) = \Hom(X,Y)$,
		\item $\id_X \in \Hom_\varepsilon(X,X)$ for all $\varepsilon > 0$, and
		\item $f \in \Hom_\varepsilon(X,Y), g \in \Hom_{\varepsilon'}(Y,Z) \Rightarrow g \comp f \in \Hom_{\varepsilon + \varepsilon'}(X, Z)$.
	\end{enumerate}
\end{definition}

The direct sum of $\varepsilon$-filtered categories is again $\varepsilon$-filtered.
However, there is no canonical way of defining a useful $\varepsilon$-filtration on the direct product of infinitely many $\varepsilon$-filtered categories.
The following definition remedies this.

\begin{definition}
	\label{def:bounded_product}
	Let $\calA_i$, $i \in I$, be a family of $\varepsilon$-filtered categories. The \emph{bounded product} $\prod^\bd_{i \in I} \calA_i$ is the subcategory of $\prod_{i \in I} \calA_i$ that has all objects, but only includes those morphisms $(f_i)_i$ that are uniformly $\varepsilon$-controlled for an $\varepsilon > 0$.
\end{definition}

The following definition describes how we can create several different $\varepsilon$-filtered categories from a given metric space.

\begin{definition}
	\label{def:controlled_categories}
	Let $G$ be a group and $(X,d)$ a locally compact metric space with isometric $G$-action.
	Let $\calA$ be an additive category with $G$-action.
	The category $\calC_G^\lf(X;\calA)$ is defined as follows.
	\begin{description}
		\item[Objects] are triples $(S, \pi: S \rightarrow X, M: S \rightarrow \Ob(\calA))$ where $S$ is a free $G$-set such that $\pi$ and $M$ are $G$-equivariant and for all cocompact $G$-subsets $K \subseteq X$ the preimage $\pi^{-1}(K)$ is cofinite, i.e.\ has finite quotient.
		\item[Morphisms] $\varphi: (S, \pi, M) \rightarrow (S', \pi', M')$ are given by matrices $(\varphi_{s'}^s)_{s \in S, s' \in S'}$, where $\varphi_{s'}^s: M(s) \rightarrow M'(s')$ is a morphism in $\calA$ such that $g.\varphi_{s'}^s = \varphi_{gs'}^{gs}$ and where
		\begin{enumerate}
			\item \label{c_rows_columns_finite} all rows and columns are finite, and
			\item \label{c_bounded_x} $\exists \alpha > 0 \,\forall s, s' : \varphi_{s'}^s \neq 0 \Rightarrow d(\pi_X(s), \pi'_X(s')) < \alpha$.
		\end{enumerate}
		Composition is given by matrix multiplication, i.e.\ $(\varphi \comp \psi)^s_{s''} := \sum\limits_{s'} \varphi^{s'}_{s''} \comp \psi^s_{s'}$.
	\end{description}
	A morphism $\varphi: (S, \pi, M) \rightarrow (S', \pi', M')$ in this category is \emph{$\varepsilon$-controlled} if 
	$$\forall s \in S, s' \in S': \varphi_{s'}^s \neq 0 \Rightarrow d(\pi_X(s), \pi'_X(s')) < \varepsilon.$$
	Based on this category we can derive various different flavours:
	\begin{itemize}
		\item Define $\calO^\lf_{G}(X; \calA)$ as the subcategory of $\calC_G^\lf(X \times \bbN; \calA)$ that consists of all objects, but only of those morphisms that satisfy the following convergence condition.
		\begin{equation*}
			\label{eq:convergence}
			\forall \varepsilon>0 \,\exists t_0>0 \,\forall s,s': \pi_\bbN(s) > t_0 \wedge \varphi_{s'}^s \neq 0 \Rightarrow d(\pi_X(s), \pi'_X(s')) < \varepsilon \tag{$\ast$}
		\end{equation*}
		\item Define $\calT_G^\lf(X; \calA)$ as the full subcategory of $\calC_G^\lf(X \times \bbN; \calA)$ consisting of those objects $(S, \pi, M)$ with $\im(\pi_\bbN)$ being finite.
		Note that this is automatically a full subcategory of $\calO_{G}^\lf(X; \calA)$ as well.
		\item Define $\calO^\lf_{G}(X \mid Y; \calA)$ as the subcategory of $\calC_G^\lf(X \times Y \times \bbN; \calA)$ that consists of all objects, but only of those morphisms that satisfy the same convergence condition (\ref{eq:convergence}) as above, i.e.\ morphisms satisfy the control condition in $X$-direction, but not in $Y$-direction.
		\item All these categories have a compactly supported version for which we demand objects $(S, \pi, M)$ to satisfy in addition
		$$\exists K \subseteq X \text{ cocompact}: \im(\pi_X) \subseteq K.$$
		We explicitly only want this condition to only apply to the $X$-direction, not the $\bbN$- or $Y$-direction. We denote this modification by dropping the $\lf$ from the notation.
		\item If $G$ is the trivial group, we simply drop it from the notation as well.
	\end{itemize}
\end{definition}

In some cases it is useful to distinguish $\varepsilon$-filtrations for each coordinate direction of the categories defined above.
This allows us to define a more refined notion of the bounded product for the category $\calO_G(- \mid -; \calA)$.

\begin{definition}
	\label{def:controlled_product}
	Let $X_n$ and $Y_n$ be two sequences of metric spaces with isometric $G_n$-actions. The \emph{controlled product} $\contr_{n \in \bbN} \calO_{G_n}(X_n \mid Y_n; \calA)$ is the subcategory of $\prod^\bd_{i \in I} \calO_{G_n}(X_n \mid Y_n; \calA)$ that has all objects, but only includes those morphisms $(\varphi_n)_n$ that satisfy the following convergence condition in $Y_n$-direction:
	$$\forall \varepsilon > 0 \,\exists N_0 \in \bbN \,\forall n \geq N_0: \varphi_n \text{ is $\varepsilon$-controlled in $Y_n$-direction.}$$
\end{definition}

\begin{remark}
	Certain maps $X \rightarrow X'$ induce functors on the categories defined above.
	The properties those maps have to satisfy depend on the category.
	For instance, uniformly continuous maps preserve (\ref{eq:convergence}).
\end{remark}

\subsection{The Conjecture}

In this article we are concerned with the following conjecture which is the algebraic analogue of the coarse Baum--Connes conjecture.
It already appeared implicitly in \cite{Carlsson.Pedersen.1995} where the authors studied a splitting of the $K$-theoretic assembly map.

\begin{conjecture}
	\label{conj:main_conjecture}
	Let $X$ be a locally compact metric space. 
	The coarse assembly map
	$$\colim_{d \geq 0} H_*^\lf\big(P_d(X); \mathbb K_\calA^\text{alg}\big) \rightarrow K_*\big(\calC^\lf( X; \calA)\big)$$
	is an isomorphism.
\end{conjecture}

Here, $H^\lf$ is a locally finite homology theory, cf.\ \cite{Weiss.2002}, and $P_d(X)$ is the Rips complex of $X$.
This locally finite homology theory can be modelled as the $K$-theory of the quotient of $\calO^\lf$ and $\calT^\lf$.
The assembly map is then given by the connective homomorphism of a long exact sequence associated to $\calO^\lf$ and $\calT^\lf$.
More details can be found in \cite[Section 3]{Zeggel.2021}.
The important fact here is that the conjecture can be proven for $X$ by showing that $\colim_{d \geq 0} K_*\big(\calO^\lf(P_d(X); \calA)\big)$ vanishes.
For instance, Ramras, Tessera and Yu do this in \cite{Ramras.2018} in order to prove
the conjecture for spaces $X$ of finite decomposition complexity.

In this article we make use of the geometric properties of spaces of graphs with large girth and are able to get rid of the Rips complex and the colimit.
This eliminates one degree of complexity and makes it easier to prove the conjecture for those spaces.
	\section{Getting Rid of the Rips Complex}
\label{sec:getting_rid_of_the_rips_complex}
	
\begin{lemma}
	\label{lem:rips_complex_of_tree}
	The Rips complex $P_d(\Gamma)$ of a graph $\Gamma$ with $d < \frac12 \cdot \girth(\Gamma)$ is homotopy equivalent to $\Gamma$ itself via a proper and uniformly continuous homotopy.
\end{lemma}

\begin{proof}
	Let $\sigma \subseteq P_d(\Gamma)$ be a simplex.
	By definition, $\sigma$ is spanned by a set $\{v_0, \dots, v_k\}$ of vertices of diameter $\leq d$.
	Let $(t_0, \dots, t_n) \in \Delta^k$ represent a point $x \in \sigma$.
	Define $r: P_d(\Gamma) \rightarrow \Gamma$ via
	$$ r(x) := \argmin_{y \in \Gamma} \sum_{i=0}^{k} t_i d(y, v_i)^2. $$
	Since $\sigma$ is compact, the sum has a minimum and we will see that its argument is unique.
	Also, $r$ is clearly a retract of the inclusion $\Gamma \subseteq P_d(\Gamma)$.
	We will show that $r$ is continuous and therefore, as defined on simplices, uniformly continuous.
	The homotopy $H: P_d(\Gamma) \times [0,1] \rightarrow P_d(\Gamma)$ given by
	$$ H(x,t) = x + (1-t) r(x) $$
	is well-defined (for all $x \in \sigma \subseteq P_d(\Gamma)$ there is a simplex $\overline \sigma \subseteq P_d(\Gamma)$ containing both $x$ and $r(x)$, see below).
	It homotopes $r$ to $\id_{P_d(\Gamma)}$.
	
	\vspace{1ex}
	\emph{Uniqueness} of $\argmin_{y \in \Gamma} \sum_{i=0}^{k} t_i d(y, v_i)^2$.
	Assume there are two points $y',y'' \in \Gamma$ minimizing $\sum_{i=0}^{k} t_i d(y, v_i)$.
	Because of the girth condition there is a unique geodesic in $\Gamma$ connecting $y'$ and $y''$.
	Let $y \in \Gamma$ be the point in the middle of the geodesic.
	It is easy to see that
	$$ 2 \cdot t_i \underset{A}{\underbrace{d(y, v_i)}}^2 \leq t_i \underset{B}{\underbrace{d(y', v_i)}}^2 + t_i \underset{C}{\underbrace{d(y'', v_i)}}^2.$$

	This situation can be depicted as follows, where $A = a + b''$, $B = a+b'$, $C = a+b''+b$.
	\begin{center}
		\begin{picture}(150,40)
		\put(0,0){\includegraphics{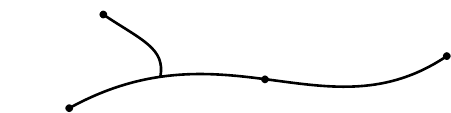}}
		\put(15,-10){$y'$}
		\put(73,-2){$y$}
		\put(131,7){$y''$}
		\put(25,32){$v_i$}
		\put(46,20){$a$}
		\put(32,-4){$b'$}
		\put(55,-1){$b''$}
		\put(100,10){$b$}
		\end{picture}
	\end{center}
	\begin{align*}
 2(a+b'')^2 &\leq (a+b')^2 + (a+b''+b)^2 \\
	\Leftrightarrow \quad\quad\quad 0 &\leq 4ab' + 2b''b' + b^2 + (b'')^2
	\end{align*}
	
	We conclude that $y$ minimizes $\sum_{i=0}^{k} t_i d(y, v_i)$ as well and that $b = 0$ because otherwise $y'$ and $y''$ would not have been minimizing in the first place.
	Thus, $y' = y = y''$.
	
	\vspace{1ex}
	\emph{Continuity} of $r(x)$. 
	Note that \vspace{-1ex}
	$$r(x) = \argmin \bigg\{ \sum_{i=0}^{k} t_i d(y, v_i)^2 \,\bigg|\, y \in \tau \subseteq \Gamma \text{ simplex}\bigg\},$$
	meaning that we can determine the minimizers for all simplices (i.e. edge or vertex) of the graph and then pick the (unique) minimizing simplex.
	If the simplex is an edge $\tau = \langle v_a, v_b \rangle$, then the value $r_\tau(x) = \argmin_{y \in \tau} \sum_{i=0}^{k} t_i d(y, v_i)^2$ can be calculated using basic analysis:
	Points in $\tau$ are parametrized by $s \cdot v_a + (1-s) \cdot v_b$.
	To find the minimizing parameter $s_0$, consider the derivative of $\sum_{i=0}^k t_i d(y,v_i)^2$ using the parametrization.
	As the zero of the first derivative of 
	$$ s \mapsto \sum_{v_i \text{ near } v_a} t_i(d(v_i, v_a) + s)^2 \;\;\; + \sum_{v_i \text{ near } v_b} t_i(d(v_i, v_b) + 1 - s)^2$$
	we get \vspace{-1ex}
	\begin{equation*}
	s_0 = \sum_{v_i \text{ near } v_b} t_i (d(v_i, v_b) + 1) \;\;\; - \sum_{v_i \text{ near } v_a} t_i d(v_i, v_a) \tag{$\ast$}
	\end{equation*}
	and therefore \vspace{-1ex}
	$$ r_\tau(x) = \begin{cases}
	v_a & \phantom{0 \leq\;} s_0 \leq 0 \\
	(1-s_0) \cdot v_a + s_0 v_b & 0 \leq s_0 \leq 1 \\
	v_b & 1 \leq s_0.
	\end{cases}$$
	These maps are continuous.
	If the minimizing point changes from one simplex to another, it must happen on a common vertex.
	Otherwise, because of the continuity of the $r_\tau$, there would be two distinct minimizers contradicting the uniqueness.
	
	\vspace{1ex}
	\emph{Common simplex for $x$ and $r(x)$}.\\
	Consider $x \in \sigma = \langle v_0, \dots, v_k \rangle \subseteq P_d(\Gamma)$.
	For all $0 \leq i \leq k$, there is $0 \leq j \leq k$ such that $r(x)$ lies on the geodesic connecting $v_i$ and $v_j$.
	Otherwise we would have the situation depicted in the following diagram, which clearly shows that $r(x)$ is not optimal in that case.
	\begin{center}
		\vspace{4em}
		\begin{tikzpicture}
			\put(0,0){\includegraphics{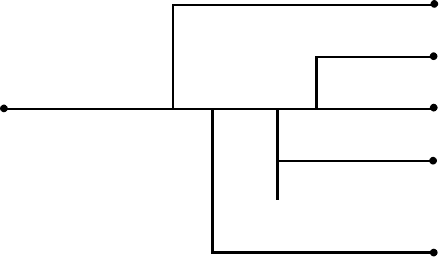}}
			\put(-10,50){$r(x)$};
			\put(129,70){$v_0$};
			\put(129,55){$v_1$};
			\put(129,40){$v_2$};
			\put(129,25){$v_3$};
			\put(90,14){$\dots$};
			\put(129,0){$v_k$};
			\draw[line cap=round, line width = 0.75pt,decorate,decoration={brace,amplitude=5pt}]
			(1.8,1.4) -- node[below,inner sep=6pt] {\(\scriptstyle \text{too much} \)} (0,1.4);
		\end{tikzpicture}~~~~~~~~~~~~~~~~~~~~~~~~
		\vspace{2em}
	\end{center}			
	
	All vertices of the geodesic $v_i \rightarrow v_j$ together with $\{v_0, \dots, v_k\}$ have diameter $\leq d$ and therefore span a simplex $\overline \sigma$ containing both $x$ and $r(x)$.
\end{proof}

\begin{remark}
	\label{rem:lipschitz}
	Equation ($\ast$) shows that the map $r: P_d(\Gamma) \rightarrow \Gamma$ is somewhat linear.
	Its operator norm is $\max \{d(v_i,v_a) + 1\} \cup \{d(v_i, v_b)\} \leq d + 1$.
	For points $x$ and $y$ in different simplices of $P_d(\Gamma)$ of distance $d(x,y) > 1$, we have $d(r(x), r(y)) \leq d \cdot d(x,y)$ because the set of vertices of a simplex in $P_d(\Gamma)$ has diameter $\leq d$.
	We conclude that $r$, and therefore $H$, are Lipschitz with Lipschitz-constant $d+1$.
\end{remark}

\begin{definition}
	Let $\calC$ be an additive category and $\calA \subseteq \calC$ a full additive subcategory. The \emph{quotient category} $\bigslant\calC\calA$ is the category that has the same objects as $\calC$, but
	$$\Hom_{\calC/\calA}(X,Y) := \bigslant{\Hom_\calC(X,Y)}{\sim}$$
	as morphism groups, where $f \sim g$ if and only if $f-g$ factors through an object in $\calA$, i.e.\ if there exists a commutative diagram in $\calC$ \vspace{-1em}
	\begin{center}
		\begin{tikzcd}[row sep=1em, column sep=1.5em]
		X \arrow[rr, "f-g"] \arrow[dr] && Y \\ & A \arrow[ur]
		\end{tikzcd}
	\end{center}
	with $A \in \calA$. This defines an additive category via $[f] \comp [g] := [f \comp g]$ and $[f] + [g] := [f + g]$. 
	Also, the canonical functor $\pr: \calC \rightarrow \bigslant{\calC}{\calA}$ is compatible with direct sums (cf. \cite{Ramras.2018}).
\end{definition}

\begin{notation}
	For $\varepsilon$-filtered categories $\calA_i$ we use the abbreviation $$\frac{\prod^\bd_i}{\bigoplus_i} \calA_i  := \bigslant{\prod^\bd_i \calA_i}{\bigoplus_i \calA_i}.$$
	Similarly, given a metric space $X$ with isometric $G$-action, we write $\calO^\gg_{G}(X; \calA)$ for the quotient $\bigslant{\calO_{G}^\lf(X; \calA)}{\calO_{G}(X; \calA)}$.
\end{notation}

\begin{proposition}
	\label{prop:goal_without_rips}
	Let $\Gamma_n$ be a sequence of finite graphs with large girth and bounded geometry, $\widetilde\Gamma_n$ its sequence of universal covers and $G_n := \pi_1(\Gamma_n)$ their fundamental groups.
	Conjecture \ref{conj:main_conjecture} for $X = \bigsqcup_{n \in \bbN} \Gamma_n$ is equivalent to
	$$ K_*\left(\frac{\prod^\bd_{n\in \bbN}}{\bigoplus_{n \in \bbN}} \calO_{G_n}(\widetilde\Gamma_n; \calA)\right) = 0.$$
\end{proposition}

\begin{proof}
	For a fixed $d$ there is a homotopy equivalence $P_d(\Gamma_n) \simeq \Gamma_n$ for all $n \geq N_0$ for some $N_0 \in \bbN$.
	By Remark \ref{rem:lipschitz} we know that these homotopies are Lipschitz with Lipschitz constant independent of $n$.
	This implies $\bigsqcup_{n \geq N_0} P_d(\Gamma_n) \simeq \bigsqcup_{n \geq N_0} \Gamma_n$, yielding
	\begin{align*}
	K_*\big(\calO^\gg(\bigsqcup\nolimits_{n \in \bbN} P_d(\Gamma_n); \calA)\big) 
	&\cong K_*\big(\calO^\gg(\bigsqcup\nolimits_{n \geq N_0} P_d(\Gamma_n); \calA)\big) \\
	&\cong K_*\big(\calO^\gg(\bigsqcup\nolimits_{n \geq N_0} \Gamma_n; \calA)\big) 
	\cong K_*\big(\calO^\gg(\bigsqcup\nolimits_{n \in \bbN} \Gamma_n; \calA)\big).
	\end{align*}
	This and \cite[Lemma 4.8]{Zeggel.2021} imply
	\todo[inline]{How do I cite my unpublished paper?}
	$$K_*\big(\calO^\gg(P_d(X); \calA)\big) \cong K_*\big(\calO^\gg(X; \calA)\big) \cong K_*\left(\frac{\prod^\bd_{n\in \bbN}}{\bigoplus_{n \in \bbN}} \calO_{G_n}(\widetilde\Gamma_n; \calA)\right)$$
	for every $d \geq 0$ which proves the claim.
	In order to use \cite[Lemma 4.8]{Zeggel.2021}, we must show that the sequence of universal covering maps $p_n: \widetilde\Gamma_n \rightarrow \Gamma_n$ is asymptotically faithful, i.e.\ there is a sequence of numbers $R_n > 0$ with $R_n \uset{n \rightarrow \infty}{\longrightarrow} \infty$ such that
	$$p_n|_{B_{R_n}(x)} : B_{R_n}(x) \rightarrow \Gamma_n$$
	is an isometric embedding for all $n \in \bbN$ and $x \in \widetilde \Gamma_n$.
	In this particular case we can choose $R_n := \frac{\girth(\Gamma_n)}{4}$.
\end{proof}
	\section{Proof Outline}
\label{sec:proof_outline}
	
Let $\Gamma_n$ be a sequence of finite graphs with bounded geometry and large girth.
Let $\widetilde\Gamma_n$ be the universal covering of $\Gamma_n$ and $G_n := \pi_1(\Gamma_n)$ its fundamental group.

In order to prove Theorem \ref{thm:main_theorem} we show $K_*\left(\frac{\prod^\bd_{n\in \bbN}}{\bigoplus_{n \in \bbN}} \calO_{G_n}(\widetilde\Gamma_n; \calA)\right) = 0$.
We do this by showing that the identity factors through zero.
Using a transfer map and long-and-thin covers we construct the following commutative diagram.
\begin{center}
	\label{dia:proof_outline}
	\begin{tikzcd}[column sep=1.5em]
		K_*\left(\frac{\prod^\bd_{n\in \bbN}}{\bigoplus_{n \in \bbN}} \calO_{G_n}(\widetilde\Gamma_n \mid (V_n \times \overline{\widetilde\Gamma}_n, d_{C(n)}) ; \calA)\right)
		\arrow[r, "f_*"] \arrow[dr, "\Pr"] &
		K_*\left(\frac{\contr_{n\in\bbN}}{\bigoplus_{n\in\bbN}} \calO_{G_n}(\widetilde\Gamma_n \mid \Sigma_n ; \calA)\right) \arrow[d, "\Pr"] & \hspace{-2.5em}=0 \\
		K_*\left(\frac{\prod^\bd_{n\in \bbN}}{\bigoplus_{n \in \bbN}} \calO_{G_n}(\widetilde\Gamma_n; \calA)\right) \arrow[u, "\trans_*"] \arrow[r, "\id"] &
		K_*\left(\frac{\prod^\bd_{n\in \bbN}}{\bigoplus_{n \in \bbN}} \calO_{G_n}(\widetilde\Gamma_n; \calA)\right)
	\end{tikzcd}
\end{center}
Our roadmap is clear.
We have to do the following steps:
\begin{enumerate}
	\item Define metric $d_{C(n)}$ on the space $V_n \times \overline{\widetilde\Gamma}_n$. 
	(See Definition \ref{def:equivariant_metric}. The numbers $C(n)$ are chosen based on Lemma \ref{lem:thicken_thin_cover}.)
	\item Define the spaces $\Sigma_n$ and the functor $f_*$ such that the upper right triangle commutes. 
	(See Corollary \ref{cor:upper_right_triangle}.)
	\item Define $\trans_*$ such that the lower left triangle commutes.
	(See Lemma \ref{lem:transfer_commutatve}.)
	\item Show $K_*\bigg(\frac{\contr_{n\in\bbN}}{\bigoplus_{n\in\bbN}} \calO_{G_n}(\widetilde\Gamma_n \mid \Sigma_n ; \calA)\bigg) = 0$.
	(See Lemma \ref{lem:trivial_k_theory} and Lemma \ref{lem:proof_of_assumption}.)
\end{enumerate}
	
	\section{The Metric (Step 1)}

Again, we fix a sequence of finite graphs $\Gamma_n$ with bounded geometry and large girth.
Let $V_n \subseteq \Gamma_n$ be the set of vertices.
The universal cover $\widetilde\Gamma_n$ of $\Gamma_n$ is also a graph and we denote its set of vertices with $\widetilde{V}_n$.
As before, we define $G_n := \pi_1(\Gamma_n)$.
Since $\widetilde{\Gamma}$ is a tree, we can compactify it by adding its boundary at infinity.

\begin{definition}[Gromov boundary]
	\label{def:boundary}
	Let $T = (V,E)$ be a tree and $v \in V$. Its \emph{boundary} is given by 
	$$\partial T = \{(v_n)_{n\in\bbN} \where v = v_0 \text{, } (v_{i-1}, v_i) \in E \text{ and } v_{i-1} \neq v_{i+1}\}$$
	i.e.\ the set of all (local hence global) geodesics starting at $v$.
	
	For $\zeta = (v_n)_n, \xi = (w_n)_n \in \partial T$ define $(\zeta, \xi)_v := \sup \{n \in \bbN \where v_n = w_n \}$.
	This induces a metric $d_v$ on $\partial T$ by
	$$d_v(\zeta, \xi) := e^{-(\zeta, \xi)_v}.$$
	This metric can be extended to a metric on $V \cup \partial T$ by regarding $V$ as the set
	$$V = \{(v_n)_{n\in\bbN} \where v = v_0 \text{, } \exists k \in \bbN\, \forall i \leq k\!: [ (v_{i-1}, v_i) \in E \,\wedge\, v_{i-1} \neq v_{i+1} ], \forall i > k\!: v_{i-1} = v_i \}$$
	i.e.\ the set of all geodesics that are eventually constant. These eventually constant geodesics can be identified with their endpoints. The product $(\xi, \zeta)_v$ then has the exact same definition for $\zeta, \xi$ in $V \cup \partial T$.
\end{definition}

\begin{definition}
	\label{def:equivariant_metric}
	Choose a $G_n$-equivariant map $q_n: \widetilde V_n \rightarrow G_n$ such that the diameter of $q_n^{-1}(1)$ is bounded from above by $2 \cdot \diam(\Gamma_n)$.
	For any $C > 0$ and any metric $d_{\overline {\widetilde{\Gamma}}_{n}}$ on $\overline {\widetilde{\Gamma}}_{n}$, we define the metric $d_C$ on $\widetilde V_n \times \overline {\widetilde{\Gamma}}_{n}$ via\footnote{Please note that both $q_n(...)^{-1}$ and $q_n^{-1}(...)$ appear in this article and should not be confused.}
	$$ ((v, \xi), (w, \zeta)) \mapsto \inf \sum_{i=1}^m C \cdot d_{\overline{\widetilde{\Gamma}}_n}(q_n(v_i)^{-1}. \xi_{i-1}, q_n(v_i)^{-1}. \xi_i) + d_{\widetilde V_n}(v_{i-1}, v_i)$$
	where the infimum is taken over all finite sequences 
	$$(v, \xi) = (v_0, \xi_0), (v_1, \xi_1), \dots, (v_{m-1}, \xi_{m-1}), (v_m, \xi_m) = (w, \zeta).$$
	For $v \in \widetilde V_n$ we define the (non-equivariant) metric $d_v$ on $\overline{\widetilde{\Gamma}}_n$ via
	$$d_v(\xi, \zeta) := d_C((v, \xi), (v, \zeta)).$$
\end{definition}

\begin{remark}
	\label{rem:equivariant_metric}
	This metric is indeed a metric. It satisfies the triangle inequality, more or less by definition, and to see that it is symmetric, do the following: If $(v_0, \xi_0), \dots, (v_m, \xi_m)$ is a sequence determining $d_C((v, \xi), (w, \zeta)$, then $(v_0', \xi_0'), \dots, (v_{m+1}', \xi_{m+1}')$, given by
	$$v_i' = \begin{cases}
	v_{m-i+1} & i > 0 \\
	v_m & i = 0
	\end{cases} \quad \text{ and } \quad \xi_i' := \begin{cases}
	\xi_{m-i} & i \leq m \\
	\xi_0 & i = m+1
	\end{cases}$$
	shows that $d_C((w, \zeta), (v, \xi)) \leq d_C((v, \xi), (w, \zeta)$. Symmetry reasons imply equality. The following facts can be seen immediately.
	\begin{enumerate}
		\item The metric $d_C$ is $G_n$-invariant.
		\item $d_{\widetilde V_n}(v,w) \leq d_C((v, \xi), (w, \zeta))$ for all $v,w \in \widetilde V_n$ and $\xi, \zeta \in \overline{\widetilde{\Gamma}}_n$.
		\item $d_{\widetilde V_n}(v,w) = d_C((v, \xi), (w, \xi))$ for all $v,w \in \widetilde V_n$ and $\xi \in \overline{\widetilde{\Gamma}}_n$.
	\end{enumerate}
	From 3.\ and the triangle inequality we easily deduce
	\begin{enumerate}
		\setcounter{enumi}{3}
		\item $d_v(\xi, \zeta) \leq d_w(\xi, \zeta) + 4 \cdot \diam(\Gamma_n)$ for all $v,w \in \widetilde V_n$ with $q_n(v) = q_n(w)$ and $\xi, \zeta \in \overline{\widetilde{\Gamma}}_n$.
	\end{enumerate}
\end{remark}

\begin{lemma}[Cf. Lemma 5.1 in \cite{Bartels.2007}]
	\label{lem:thicken_thin_cover}
	Let $\alpha \geq 1$. Suppose that $\calU$ is a free $G_n$-invariant cover of $\widetilde V_n \times \overline{\widetilde{\Gamma}}_n$ such that for all $(v, \xi) \in \widetilde V_n \times \overline{\widetilde{\Gamma}}_n$ there is $U \in \calU$ such that $B_\alpha(v) \times \{\xi\} \subseteq U$. Then there exists a constant $C = C(\calU) > 1$ such that the following holds:
	
	For every $(v,\xi) \in \widetilde V_n \times \overline{\widetilde{\Gamma}}_n$ there exists $U \in \calU$ such that the open $\alpha$-ball with respect to the metric $d_C$ around $(v, \xi)$ lies in $U$.
\end{lemma}

\begin{proof}
	The proof is the same as that of \cite[Lemma 5.1]{Bartels.2007}. There, the group $G$ acts on itself being a metric space with word metric. Here, we separate these two roles by replacing the metric space $G$ by $\widetilde V_n$ and having $G_n$ act on it.
	
	For every $\xi \in \overline{\widetilde{\Gamma}}_n$, we can find $U_\xi \in \calU$ with $B_\alpha(v_0) \times \{\xi\} \subseteq U_\xi$, where $v_0 \in \widetilde V_n$ is a choice of base point with $q_n(v_0) = 1$. Choose for $v \in B_\alpha(v_0)$ an open neighbourhood $V_{v, \xi}$ of $\xi \in \overline{\widetilde{\Gamma}}_n$ such that $\{v\} \times V_{v, \xi} \subseteq U_\xi$. Put $V_\xi := \bigcap_{v \in B_\alpha(v_0)} V_{v,\xi}$. Then $\{V_\xi \where \xi \in \overline{\widetilde{\Gamma}}_n \}$ is an open cover of the compact metric space $(\overline{\widetilde{\Gamma}}_n, d_{\overline{\widetilde{\Gamma}}_n})$. Let $\varepsilon > 0$ be a Lebesgue number for this open cover.
	
	Since $\overline{\widetilde{\Gamma}}_n$ is compact, the map $\overline{\widetilde{\Gamma}}_n \rightarrow \overline{\widetilde{\Gamma}}_n$, $\xi \mapsto g.\xi$ is uniformly continuous. Hence, we can find $\delta(\varepsilon, g) > 0$ such that $d_{\overline{\widetilde{\Gamma}}_n}(g.\zeta, g.\xi) < \frac{\varepsilon}{\alpha}$ holds for all $\zeta, \xi \in \overline{\widetilde{\Gamma}}_n$ with $d_{\overline{\widetilde{\Gamma}}_n}(\zeta, \xi) < \delta(\varepsilon, g)$. Since there are only finitely many elements in $B_{\alpha}(q_n^{-1}(1))$, we can choose a constant $C$ such that $\frac{\alpha}{C} < \delta(\varepsilon, q_n(v))$ holds for all $v \in B_{\alpha}(q_n^{-1}(1))$. We thus get
	$$d_{\overline{\widetilde{\Gamma}}_n}(q_n(v).\zeta, q_n(v).\xi) < \frac{\varepsilon}{\alpha}  \quad \text{ for $\zeta, \xi \in \overline{\widetilde{\Gamma}}_n$ with $d_{\overline{\widetilde{\Gamma}}_n}(\zeta, \xi) < \frac{\alpha}{C}$ and $v \in B_\alpha(q_n^{-1}(1))$.}$$
	
	Because $d_C$ and the cover $\calU$ are $G_n$-invariant, it suffices to prove the claim for an element of the shape $(v_0, \xi) \in \widetilde V_n \times \overline{\widetilde{\Gamma}}_n$ with $q_n(v_0) = 1$. Let $(w, \zeta)$ be an element in the ball of radius $\alpha$ around $(v_0, \xi)$ with respect to the metric $d_C$. We want to show $(w, \zeta) \in U_{\xi'}$ for some $\xi' \in \overline{\widetilde{\Gamma}}_n$ only depending on $v_0$, $\xi$ and $C$. By definition of $d_C$, we can find a sequence of elements $(v, \xi) = (v_0, \xi_0), (v_1, \xi_1), \dots, (v_{m-1}, \xi_{m-1}), (v_m, \xi_m) = (w, \zeta)$ in $\widetilde V_n \times \overline{\widetilde{\Gamma}}_n$ such that
	$$\sum_{i=1}^m d_{\widetilde V_n}(v_{i-1}, v_i) + \sum_{i=1}^m C \cdot d_{\overline{\widetilde{\Gamma}}_n}(q_n(v_i)^{-1}.\xi_{i-1}, q_n(v_i)^{-1}.\xi_i) < \alpha.$$
	If $m=1$, we have $m \leq \alpha$ by assumption.
	If $m>1$, we can arrange $v_i \neq v_{i+1}$: Whenever $v_i = v_{i+1}$, delete the element $(v_i, \xi_i)$ from the sequence. The inequality above remains true because of the triangle inequality for $d_{\overline{\widetilde{\Gamma}}_n}$. Since $d_{\widetilde V_n}(v_{i-1}, v_i) \geq 1$, we conclude
	$$m \leq \alpha.$$
	By the triangle inequality, $d_{\widetilde V_n}(v_0, v_i) \leq \alpha$ for $1 \leq i \leq m$. In other words, $v_i \in B_\alpha(v_0) \subseteq B_\alpha(q_n^{-1}(1))$ for $1 \leq i \leq m$.
	
	We have $d_{\overline{\widetilde{\Gamma}}_n}(q_n(v_i)^{-1}.\xi_{i-1}, q_n(v_i)^{-1}.\xi_i) < \frac{\alpha}{C}$ for $1 \leq i \leq m$. We conclude that
	$$d_{\overline{\widetilde{\Gamma}}_n}(\xi_{i-1}, \xi_i) < \frac{\varepsilon}{\alpha}$$
	holds for $1 \leq i \leq m$. The triangle inequality implies, together with $m \leq \alpha$,
	$$d_{\overline{\widetilde{\Gamma}}_n}(\xi, \zeta) < \varepsilon.$$
	Now there is $\xi' \in \overline{\widetilde{\Gamma}}_n$, with $\zeta \in V_{\xi'}$, only depending on $v_0$, $\xi$ and $C$ because the Lebesgue number of $\{V_\xi \where \xi \in \overline{\widetilde{\Gamma}}_n \}$ is $\varepsilon$. 
	Since $w \in B_\alpha(v_0)$, we conclude $\zeta \in V_{\xi'} \subseteq V_{w, \xi'}$. This implies $(w, \zeta) \in U_{\xi'}$.
\end{proof}
	\section{The Upper Right Triangle (Step 2)}

\begin{definition}[{\cite[Section 4.1]{Bartels.2007}} ]
	Let $(X,d)$ be a metric space. Let $\calU$ be a finite dimensional cover of $X$ by open sets. Recall that points in the realization of the nerve $|\calU|$ are formal sums $x=\sum_{U \in \calU} x_U U$, with $x_U \in [0,1]$, such that $\sum_{U \in \calU} x_U = 1$ and the intersection of all those $U$ with $x_U \neq 0$ is non-empty, i.e., $\{U \where x_U \neq 0 \}$ is a simplex in the nerve of $\calU$. There is a map
	$$ f^\calU : X \rightarrow |\calU|, \quad x \mapsto \sum_{U \in \calU} f_U(x) U, \vspace{-1ex}$$
	where 
	$$f_U(x) = \frac{a_U(x)}{\sum_{V \in \calU} a_V(x)} \quad \text{ with } \quad a_U(x) = d(x,Z \smallsetminus U) = \inf\{d(x,u) \where u \notin U\}.$$
	This map is well-defined since $U$ is finite dimensional. If $X$ is a metric space with an isometric $G$-action, then $d$ is $G$-invariant, and if $\calU$ is a free open $G$-cover, then the map $f=f^\calU$ is $G$-equivariant and $|\calU|$ has a free $G$-action.
	If $\calU$ is locally finite (i.e.\ every set in $\calU$ intersects only finitely many other sets in $\calU$), then $|\calU|$ is locally compact.
\end{definition}

\begin{lemma}
	\label{lem:proper_map_to_nerve}
	Given any proper metric space $X$ and locally finite open cover $\calU$ of $X$, the map
	$$f^\calU: X \rightarrow |\calU|$$
	is metrically proper if and only if $\,\calU$ consists solely of bounded sets.
\end{lemma}

\begin{proof}
	In order to check that $f$ is metrically proper, it suffices to study preimages of simplices of $|\calU|$.
	A simplex of $|\calU|$ is hit exactly by the points in the open sets spanning said simplex.
	As all those finitely many sets are bounded, so is their union.
	
	Vice versa, if one set in $\calU$ is unbounded, then the closed star of this set in $|\calU|$ has an unbounded preimage and $f$ cannot be metrically proper.
\end{proof}

\begin{lemma}
	\label{lem:bounded sets_in_cover}
	Let $X$ be a metric space with a cocompact, isometric $G$-action and a locally finite free $G$-invariant cover $\calU$ of open sets.
	Then $\calU$ consists solely of bounded sets.
\end{lemma}

\begin{proof}
	Assume there is an unbounded set $U \in \calU$.
	Choose a sequence $(x_n)_n$ of points in $U$ that is unbounded.
	Since $X$ is cocompact, the sequence converges in $X/G$.
	Phrased differently, there exists a sequence $(g_n)_n$ in $G$ such that $g_n x_n \uset{n \rightarrow \infty}{\longrightarrow} x_\infty \in X$.
	
	\vspace{1ex}
	Now there is an (open) $V \in \calU$ such that $x_\infty \in V$ and therefore $g_nU \cap V \neq \emptyset$ for all $n \gg 0$.
	The set $\{g_n\}_n$ is infinite because otherwise $(x_n)_n$ would have been a bounded sequence.
	As $G$ acts freely on $\calU$, the set $\{g_n U\}_n$ is infinite as well.
	However, this is a contradiction to the fact that $\calU$ is locally finite.
\end{proof}

\begin{proposition}
	\label{prop:contracting_map}
	Let $\calU$ be a $D$-dimensional cover of $\widetilde V_n \times \overline{\widetilde{\Gamma}}_n$ with the properties from Theorem \ref{thm:long_thin_covers} below. For all $n \in \bbN$ there is a constant $C(n) > 1$ such that the map $f^\calU: \widetilde V_n \times \overline{\widetilde{\Gamma}}_n \rightarrow |\calU|$ is metrically proper and the following properties are satisfied.
	\begin{enumerate}
		\item The dimension of $|\calU|$ is at most D.
		\item The action $G_n \curvearrowright |\calU|$ is free.
		\item $d_{C(n)}\big((v,\xi),(w,\zeta)\big) \leq \dfrac{4D\,\girth(\Gamma_n)}{7} \, \Rightarrow \, d(f^\calU(v,\xi), f^\calU(w,\zeta)) \leq \dfrac{7\, d_{C(n)}((v,\xi), (w,\zeta))}{\girth(\Gamma_n)}$.
	\end{enumerate}
\end{proposition}

\begin{proof}
	The first two properties are immediate consequences of Theorem \ref{thm:long_thin_covers}. 
	Set $\alpha := \frac{16D^2\cdot \girth(\Gamma_n)}{7}$.
	The cover $\calU$ of $\widetilde{V}_n \times \overline{\widetilde{\Gamma}}_n$ has the exact property needed in Lemma \ref{lem:thicken_thin_cover}.
	Thus, for every $(v,\xi) \in \widetilde V_n \times \overline{\widetilde{\Gamma}}_n$, there exists $U \in \calU$ such that the open $\alpha$-ball with respect to $d_{C(n)}$ around $(v, \xi)$ lies in $U$.
	
	If $d_{C(n)}\big((v,\xi), (w,\zeta)\big)) \leq \frac{4D\cdot\girth(\Gamma_n)}{7}$, then we can apply \cite[Proposition 5.3]{Bartels.2007} and see that $d\big(f^\calU(v, \xi), f^\calU(w,\xi)\big) \leq \frac{7 \cdot d_{C(n)}((v,\xi), (w,\zeta))}{\girth(\Gamma_n)}$.
	
	The map $f^\calU$ is metrically proper because of Lemma~\ref{lem:bounded sets_in_cover} and Lemma~\ref{lem:proper_map_to_nerve}.
\end{proof}

\begin{corollary}
	\label{cor:upper_right_triangle}
	The functor
	$$f: \frac{\prod^\bd_n}{\bigoplus_n} \calO_{G_n}(\widetilde\Gamma_n \mid (V_n \times \overline{\widetilde\Gamma}_n, d_{C(n)}) ; \calA) \rightarrow \frac{\contr_n}{\bigoplus_n} \calO_{G_n}(\widetilde\Gamma_n \mid \Sigma_n ; \calA)$$
	given by the product of all $\big(\id_{\widetilde\Gamma_n}, f_n\big)_*$ is well-defined and makes the diagram
	\begin{center}
		\begin{tikzcd}
		K_*\left(\frac{\prod^\bd_n}{\bigoplus_n} \calO_{G_n}(\widetilde\Gamma_n \mid (V_n \times \overline{\widetilde\Gamma}_n, d_{C(n)}) ; \calA)\right) \arrow[r, "f_*"] \arrow[dr, "\Pr_*"] 
		& K_*\left(\frac{\contr_n}{\bigoplus_n} \calO_{G_n}(\widetilde\Gamma_n \mid \Sigma_n ; \calA)\right) \arrow[d, "\Pr_*"] \\
		& K_*\left(\frac{\prod^\bd_n}{\bigoplus_n} \calO_{G_n}(\widetilde\Gamma_n; \calA)\right)
		\end{tikzcd}
	\end{center}
	commutative. (This is the upper right triangle of the diagram on page \pageref{dia:proof_outline}.)
\end{corollary}

\begin{proof}
The functor $f$ is well-defined as can be seen as follows.
For each morphism $(\varphi_n)_n$ of $\frac{\prod^\bd_n}{\bigoplus_n} \calO_{G_n}(\widetilde\Gamma_n \mid (V_n \times \overline{\widetilde\Gamma}_n, d_{C(n)}) ; \calA)$ there is an $\varepsilon > 0$ such that all $\varphi_n$ are $\varepsilon$-controlled. (Here, $\varepsilon$ may be very large.)

If $\varepsilon \leq \frac{4D\,\girth(\Gamma_n)}{7}$, or equivalently $\girth(\Gamma_n) \geq \frac{7 \varepsilon}{4D}$, then Proposition \ref{prop:contracting_map}.3 applies and shows that $(f_n)_*(\varphi_n)$ is $\frac{7 \varepsilon}{\girth(\Gamma_n)}$-controlled.
We conclude that for large $n$ the functor $f$ has the contractive behaviour as demanded by the controlled product.
We do not have to worry about the small $n$ because we quotient out the direct sum.
The map $f^\calU$ has to be metrically proper so that the induced functor preserves the local finiteness condition of the objects of $\calO_G(-|-;\calA)$.

The diagram commutes because $f$ is simply the identity on the first coordinate.
\end{proof}

\subsection{Long and Thin Covers}

\begin{theorem}
	\label{thm:long_thin_covers}
	Let $\Gamma_n$ be a sequence of finite graphs. Set $\alpha(n) := \frac{\girth(\Gamma_n)}{7}$. Then
	$$\forall n \in \bbN \; \exists\, \calU \text{ free $\pi_1(\Gamma_n)$-cover of } \widetilde V_n \times \overline{\widetilde\Gamma}_n:$$
	$$\forall (v,\xi) \in \widetilde V_n \times \overline{\widetilde\Gamma}_n \; \exists \, U \in \calU: B_{\alpha(n)}(v) \times \{\xi\} \subseteq U$$
	where the $\calU$ are at most 7-dimensional.
	This way, if the $\Gamma_n$ have large girth, then the covers become longer with $n \rightarrow \infty$.
\end{theorem}

This theorem is a consequence of Theorem \ref{thm:final_cover} below whose proof makes use of a so-called flow space.

Let $\Gamma = (V,E)$ be a finite graph.
Consider the universal cover $\widetilde{\Gamma} = (\widetilde{V}, \widetilde{E})$ and choose a base point $v_0 \in \widetilde{V}$.
The spaces $\widetilde\Gamma$ and $\widetilde{V}$ can be compactified by adding the boundary at infinity, see Definition \ref{def:boundary}.
The resulting compact spaces are denoted by $\overline{\widetilde\Gamma}$ and $\overline{\widetilde V}$.
All these spaces have a canonical $\pi_1(\Gamma)$-action.

\begin{definition}
	\label{def:coarse_flow_space}
	The \emph{flow space} $FS$ for $\Gamma$ is the subspace of $\widetilde V \times  \overline{\widetilde V} \times \overline{\widetilde V}$ 
	consisting of triples $(v, \xi_-, \xi_+)$ such that there exists a geodesic in 
	$\widetilde \Gamma$ from $\xi_-$ to $\xi_+$ that contains $v$. The flow space has the following \emph{flow} $\Phi_\tau : FS \rightarrow FS$; $(v, \xi_-, \xi_+) \mapsto (\phi_\tau(v), \xi_-, \xi_+)$, where $\phi_\tau(v)$ is the vertex on the geodesic $v \rightarrow \xi_{\sign(\tau)}$ with distance $|\tau|$ to $v$, if it exists, and $\xi_{\sign(\tau)}$, otherwise. 
\end{definition}

\begin{remark}
	Usually, the term \emph{flow} implies a property like $\Phi_{\tau + \tau'} = \Phi_{\tau} \comp \Phi_{\tau'}$.
	This is not always the case for the flow defined in Definition \ref{def:coarse_flow_space}.
	However, it is satisfied for example for points $(v, \xi_-, \xi_+)$ where $\xi_-$ and $\xi_+$ are close to the boundary. 
\end{remark}

\begin{theorem}
	\label{thm:final_cover}
	Let $\Gamma = (V,E)$ be a finite graph and set $\alpha := \frac{\girth(\Gamma)}{7}$. There exists a locally finite cover $\calU$ of $\widetilde V \times \overline{\widetilde \Gamma}$ with the following properties.
	\begin{enumerate}
		\item $\calU$ is $\pi_1(\Gamma)$-invariant, i.e.\ $g.U \in \calU$ for all $U \in \calU$ and $g \in \pi_1(\Gamma)$,
		\item $\calU$ is free, i.e.\ $g.U \cap U = \emptyset$ for all $U \in \calU$ and $g \in \pi_1(\Gamma) \smallsetminus \{1\}$,
		\item for all $(v, \xi) \in \widetilde V \times \overline{\widetilde \Gamma}$, there is $U \in \calU$ such that $B_\alpha(v) \times \{\xi\} \subseteq U$, and
		\item $\calU$ is at most 7-dimensional.
	\end{enumerate}
\end{theorem}

\begin{proof}
	The proof consists of three steps.
	
	\vspace{1ex}
	\emph{Step 1. The cover of $FS$.} \\
	By Lemma \ref{lem:cover_of_flow_space}, there is a free, 5-dimensional, $\pi_1(\Gamma)$-invariant cover $\calU_{FS}$ of $FS \subseteq \widetilde V \times \overline{\widetilde V} \times \overline{\widetilde V}$ such that for all $(v,(\xi_-, \xi_+)) \in FS$ there is $U \in \calU_{FS}$ such that $B_\alpha(v) \times \{(\xi_-, \xi_+)\} \cap FS \subseteq U$.
	All elements of $\calU_{FS}$ are bounded in $\widetilde{V}$-direction.
	
	\vspace{1ex}
	\emph{Step 2. Pulling back the cover of $FS$ to a cover of $\widetilde V \times \partial{\widetilde \Gamma}$.} \\
	Using the cover $\calU_{FS}$ of $FS$ and Lemma \ref{lem:cover_of_boundary}, there is a free, 5-dimensional, $\pi_1(\Gamma)$-invariant cover $\calU_{\partial\widetilde \Gamma}$ of $\widetilde V \times \partial \widetilde \Gamma$ such that for all $(v,\xi) \in V \times \partial \widetilde \Gamma$ there is $U \in \calU_{\partial \widetilde \Gamma}$ such that $B_\alpha(v) \times \{\xi\} \subseteq U$.
	All elements of $\calU_{\partial\Gamma}$ are bounded in $\widetilde{V}$-direction.
	
	\vspace{1ex}
	\emph{Step 3. Assembling the cover of $\widetilde V \times \overline{\widetilde \Gamma}$.} \\
	Using Lemma \ref{lem:enlarge_cover_of_boundary}, we can thicken the cover of $\partial\widetilde\Gamma$ to get $\big\{\widehat{U} \where U \in \calU_{\partial \widetilde \Gamma}\big\}$. 
	The lemma ensures that $\{\widehat{U} \where U \in \calU_{\partial \widetilde \Gamma}\}$ is free, $\pi_1(\Gamma)$-invariant, has dimension at most $\dim \calU_{\partial \widetilde \Gamma} \leq 5$, and that its elements are bounded in $\widetilde{V}$-direction.
	
	Choose a fundamental domain $F \subseteq \widetilde{V}$ for the $\pi_1(\Gamma)$-action.
	By Lemma \ref{lem:boundary_thickening}, there is an $N \in \bbN$ such that
	\begin{equation*}
	\label{eq:wide_cover_at_boundary}
	\forall (v, \xi) \in F \times \set{\xi \in \overline{\widetilde \Gamma}}{\exists \zeta \in \partial \widetilde\Gamma: (\zeta, \xi)_{v_0} \geq N} \; \exists U \in \calU_{\partial\widetilde\Gamma}: \; B_\alpha(v) \times \{\xi\} \subseteq \widehat U. \tag{$\ast$}
	\end{equation*}
	Set $R := N + \alpha + 2 \cdot \diam(\Gamma) + \frac23$.
	Then we obtain the desired cover as the union
	$$\underset{=: \, \calU_1}{\underbrace{\set{B_R(g.v_0) \times B_{\nicefrac23}(g.v)}{v \in F, g \in \pi_1(\Gamma)}}} \;\cup\; \underset{=: \, \calU_2}{\underbrace{\big\{\widehat{U} \where U \in \calU_{\partial\widetilde\Gamma}\big\}}}.$$
	
	The covers $\calU_1$ and $\calU_2$ are locally finite. Since the elements of both $\calU_1$ and $\calU_2$ are bounded, they can only intersect finitely many elements of $\calU_2$ and $\calU_1$, respectively. Therefore the union $\calU_1 \cup \calU_2$ is locally finite as well.
	
	\vspace{1ex}
	\emph{Ad 1.} Clearly, $\calU_1$ is $\pi_1(\Gamma)$-invariant.
	The union of $\pi_1(\Gamma)$-invariant sets is again $\pi_1(\Gamma)$-invariant.
	
	\vspace{1ex}
	\emph{Ad 2.} The set $\calU_1$ is free because already the set of balls $B_{\nicefrac23}(g.v)$ is acted upon freely.
	The union of free $\pi_1(\Gamma)$-sets is again free.
	
	\vspace{1ex}
	\emph{Ad 3.} By construction, $\calU_2$ already satisfies this property for points $(v, \xi) \in \widetilde{V} \times \partial\widetilde\Gamma$.
	Since $\calU_1 \cup \calU_2$ is $\pi_1(\Gamma)$-invariant, we only need to check this for points $(v, \xi) \in F \times \widetilde\Gamma$:
	Let $(v, \xi) \in F \times \widetilde\Gamma$.
	Choose $g \in \pi_1(\Gamma)$ and $v' \in F$ such that $\xi \in B_{\nicefrac23}(g.v')$.
	If $B_\alpha(v) \subseteq B_R(g.v_0)$, then $B_\alpha(v) \times \{\xi\} \subseteq B_R(g.v_0) \times B_{\nicefrac23}(g.v') \in \calU_1$.
	Otherwise, if $B_\alpha(v) \not\subseteq B_R(g.v_0)$, we argue as follows.
	$$R -\alpha \, \leq \, d(v, g.v_0) \, \leq \underset{\!\leq\; \diam(\Gamma)\!}{\underbrace{d(v,v_0)}} + d(v_0, g.v') + \underset{\leq\; \diam(\Gamma)}{\underbrace{d(g.v', g.v_0)}}$$
	$$\Rightarrow \; d(v_0, g.v') \geq R -\alpha - 2 \cdot \diam(\Gamma).$$
	Because of $\xi \in B_{\nicefrac23}(g.v')$, we conclude 
	$$d(v_0, \xi) \geq d(v_0, g.v') - \frac23 \geq R -\alpha - 2 \cdot \diam(\Gamma) - \frac23 = N.$$
	If we choose a geodesic ray $\zeta \in \partial\widetilde{\Gamma}$ starting at $v_0$ and passing $\xi$, then $(\zeta, \xi)_{v_0} \geq N$.
	Now (\ref{eq:wide_cover_at_boundary}) shows that we can find $\widehat{U} \in \calU_2$ such that $B_\alpha(v) \times \{\xi\} \subseteq \widehat{U}$.
	
	\vspace{1ex}
	\emph{Ad 4.} The dimension of $\calU_1 \cup \calU_2$ is bounded from above by
	\begin{equation*}
	\dim (\calU_1 \cup \calU_2) \leq \dim(\calU_1) + \dim(\calU_2) + 1 \leq 1 + 5 + 1 = 7.
	\end{equation*}
	Here, $\calU_1$ is 1-dimensional because $\widetilde{\Gamma}$ is 1-dimensional.
\end{proof}

\subsection{Construction of the Cover}

In this subsection we retain the assumptions of Theorem \ref{thm:final_cover}, i.e. $\Gamma = (V,E)$ is a finite graph and $\alpha = \frac{\girth(\Gamma)}{7}$. 
Also, $FS$ is the flow space for $\Gamma$ as in Definition \ref{def:coarse_flow_space}. 
Note that the girth of $\Gamma$ is a lower bound on the translation distance of the $\pi_1(\Gamma)$-action on $\widetilde \Gamma$, i.e.\ $d(v, g.v) \geq \girth(\Gamma)$ for all $v \in \widetilde \Gamma$ and $g \in \pi_1(\Gamma)$ with $g \neq 1$. 
The choice of $\alpha$ ensures that $g.B_{3\alpha}(v) \cap B_{3\alpha}(v) = \emptyset$ for all $v \in \widetilde \Gamma$ and $g \in \pi_1(\Gamma)$ with $g \neq 1$.

\begin{lemma}
	\label{lem:cover_of_flow_space}
	There is a free, 5-dimensional, $\pi_1(\Gamma)$-invariant, locally finite cover $\calU_{FS}$ of $FS \subseteq \widetilde V \times \overline{\widetilde V} \times \overline{\widetilde V}$ such that for all $(v,(\xi_-, \xi_+)) \in FS$ there is $U \in \calU_{FS}$ such that $B_\alpha(v) \times \{(\xi_-, \xi_+)\} \cap FS \subseteq U$.
	All elements of $\calU_{FS}$ are bounded in $\widetilde{V}$-direction.
\end{lemma}

\begin{proof}
	Choose a maximal $2\alpha$-separated subset $V_0$ of $V$. Lift this set to $\widetilde V_0 \subseteq \widetilde V$. This way, $\widetilde V_0$ is $\pi_1(\Gamma)$-invariant and $2 \alpha$-separated.
	
	For $v \in \widetilde V$ and $\xi \in \overline{\widetilde V}$ set
	$$ U(v, \xi) := \{ \zeta \in \overline{\widetilde V} \where (\zeta, \xi)_v \geq 6 \alpha\},$$
	where $(\zeta, \xi)_v$ is the Gromov product, see Definition \ref{def:boundary}.
	This is an open neighbourhood of $\xi$ and $U(v, \xi) = U(v, \zeta)$ for all $\zeta \in U(v, \xi)$. It is also compatible with the $\pi_1(\Gamma)$-action: $g.U(v, \xi) = U(g.v, g.\xi)$.
	
	We set 
	$$\calU_{FS} := \big\{ B_{3 \alpha}(v) \times U(v, \xi_-) \times U(v, \xi_+) \cap FS \where v \in \widetilde V_0,\xi_\pm \in \overline{\widetilde V} \big\}.$$
	Clearly, $\calU_{FS}$ is locally finite and $\pi_1(\Gamma)$-invariant. Since $g.B_{3\alpha}(v) \cap B_{3\alpha}(v) = \emptyset$, we have $g.U \cap U = \emptyset$ for all $U \in \calU_{FS}$. If $v \in \widetilde V$ is arbitrary, then there is $v_0 \in \widetilde V_0$ with $d(v,v_0) \leq 2\alpha$. Otherwise, $V_0$ would not be maximal $2\alpha$-separated. Hence $B_\alpha(v) \subseteq B_{3\alpha}(v_0)$. Thus, for all $(v,(\xi_-, \xi_+)) \in FS$, there is $U \in \calU_{FS}$ such that $B_\alpha(v) \times \{(\xi_-, \xi_+)\} \cap FS \subseteq U$.
	
	Finally, we consider distinct sets $U_1, \dots, U_n \in \calU_{FS}$ with $U_1 \cap \dots \cap U_n \neq \emptyset$. In this case, there are $v_1, \dots, v_n \in \widetilde V_0$ such that $B_{3 \alpha}(v_1) \cap \dots \cap B_{3 \alpha}(v_n) \neq \emptyset$. These elements are distinct as well because, for fixed $v$, the $U(v, \xi)$ are either disjoint or equal.
	Without loss of generality we assume that $v_1, \dots, v_n \in B_{6 \alpha}(v_1)$. By construction,
	$$B_{6\alpha}(v) \times U(v, \xi_-) \times U(v,\xi_+) \cap FS \subseteq \big(B_{6\alpha}(v) \cap (\xi_- \rightarrow \xi_+) \big) \times U(v, \xi_-) \times U(v,\xi_+) \cap FS.$$
	A $2\alpha$-separated subset of a $6\alpha$-ball inside a geodesic can consist of at most 6 points. Hence $n \leq 6$ and therefore $\dim \calU_{FS} \leq 5$.
\end{proof}

\begin{lemma}
	\label{lem:cover_of_boundary}
	Assume there is a free, $n$-dimensional, $\pi_1(\Gamma)$-invariant, locally finite cover $\calU_{FS}$ of $FS \subseteq \widetilde V \times \overline{\widetilde V} \times \overline{\widetilde V}$ such that for all $(v,\xi_-, \xi_+) \in FS$ there is $U \in \calU_{FS}$ such that $B_\alpha(v) \times \{(\xi_-, \xi_+)\} \cap FS \subseteq U$.
	Then there is a free, $n$-dimensional, $\pi_1(\Gamma)$-invariant, locally finite cover $\calU_{\partial\widetilde\Gamma}$ of $\widetilde V \times \partial \widetilde \Gamma$ such that for all $(v,\xi) \in \widetilde V \times \partial \widetilde \Gamma$ there is $U \in \calU_{\partial \widetilde \Gamma}$ such that $B_\alpha(v) \times \{\xi\} \subseteq U$.
	
	If all elements of $\calU_{FS}$ are bounded in $\widetilde{V}$-di\-rection, then all elements of $\calU_{\partial\Gamma}$ are bounded in $\widetilde{V}$-direction as well.
\end{lemma}

\begin{proof}
	This proof is based on the argumentation in \cite[page 5]{Bartels.2015}.
	For $W \subseteq FS$ and $\tau > 0$ define 
	$$\iota^{-\tau}(W) := \big\{ (v, \xi) \in \widetilde V \times \partial \widetilde \Gamma \where \Phi_\tau(v,v,\xi) \in W \big\}.$$
	We show that there is $\tau > 0$ such that $\iota^{-\tau}\calU_{FS}$ is the desired cover of $\widetilde V \times \partial \widetilde \Gamma$. Note that $g.(\iota^{-\tau}U) = \iota^{-\tau}(g.U)$. This immediately implies that $\iota^{-1}\calU_{FS}$ is a free, $G$-invariant cover. It is also clear that $\dim \iota^{-\tau} \calU_{FS} \leq \dim \calU_{FS}$. 
	
	If $W$ is bounded in $\widetilde{V}$-direction, then $\iota^{-\tau}(W)$ is bounded in $\widetilde{V}$-direction as well because, for an element $(v, v, \xi) \in \iota^{-\tau}(W)$ and the corresponding point $\Phi_\tau(v, v, \xi) = (\phi_\tau(v), v, \xi) \in W$, we have $d(v, \phi_\tau(v)) = \alpha$.
	
	Next, we mention that $\iota^{-\tau} (W)$ is open in $\widetilde V \times \overline{\widetilde \Gamma}$, provided that $W$ is open in $FS$, as it is the preimage under a continuous map.
	
	The only thing left to prove is that for $(v, \xi) \in \widetilde V \times \partial \widetilde \Gamma$ there is $U \in \calU_{\partial \widetilde \Gamma}$ such that $B_\alpha(v) \times \{\xi\} \subseteq U$.
	Assume the contrary, i.e.\ for all $\tau > 0$ there are $(v_\tau, \xi_\tau) \in \widetilde V \times \overline{\widetilde \Gamma}$ such that for all $U \in \calU_{\partial \widetilde \Gamma}$ we have $B_\alpha(v_\tau) \times \{\xi_\tau\} \not\subseteq U$. Set $w_\tau := \Phi_\tau(v_\tau, v_\tau, \xi_\tau)$. Since the cover is $\pi_1(\Gamma)$-equivariant we can assume that all $v_\tau$ lie in the same (finite) fundamental domain of the $\pi_1(\Gamma)$-action. Consequently, we can find a subsequence where $w_\tau$ is constant, say with constant value $w$.
	
	Since $\overline{\widetilde V}$ is compact, there is a subsequence such that $v_\tau$ converges against a point $\xi_-$ which is necessarily on the boundary $\partial \widetilde \Gamma$ because $d(v_\tau, w) = \tau$ increases indefinitely. 
	
	Once again, we choose a subsequence such that $\xi_\tau$ converges inside $\partial \widetilde \Gamma$ against the boundary point $\xi_+$.
	Using the assumptions on $\calU_{FS}$ for $(w, \xi_-, \xi_+)$, we obtain $V \in \calU_{FS}$ such that $B_\alpha(w) \times \{(\xi_-, \xi_+)\} \cap FS \subseteq V$. 
	We even know that $B_\alpha(w) \times V' \cap FS \subseteq V$ for some open subset $V' \subseteq \overline{\widetilde V} \times \overline{\widetilde V}$. If we choose $\tau$ large enough, then we can achieve $B_\alpha(v_\tau) \times \{\xi_\tau\} \subseteq V'$ and therefore get $B_\alpha(w) \times B_\alpha(v_\tau) \times \{\xi_\tau\} \subseteq V$.
	
	We want to show $B_\alpha(v_\tau) \times \{\xi_\tau\} \subseteq \iota^{-\tau} (V)$ for $\tau \gg 0$, which is the contradiction to the assumption in the beginning. We must therefore show
	$$\forall v \in B_\alpha(v_\tau): \Phi_\tau(v,v, \xi_\tau) \in V.$$
	Yet this is clear now, since $d(w, \Phi_\tau(v,v, \xi_\tau)) < \alpha$.
\end{proof}

\begin{lemma}
	\label{lem:enlarge_open_sets_from_boundary}
	Let $T$ be a tree with boundary $\partial T$. There is a map
	\begin{align*}
	\{\text{open subsets of } \partial T \} &\rightarrow \{ \text{open subsets of } \overline{T} \} \\
	U &\mapsto \widehat{U}
	\end{align*}
	such that $U \subseteq \widehat{U}$ and $\widehat{U} \cap \widehat{V} \neq \emptyset \Rightarrow U \cap V \neq \emptyset$.
\end{lemma}

\begin{proof}
	Let $U \subseteq \partial T$ be an open subset, i.e.\ $U$ is the union of open balls:
	$$U = \bigcup_{i \in I} B_i.$$
	Open balls are given by $B_i = \{ \xi \in \partial T \where (\xi, \xi_i)_v \geq N_i\}$ for some $\xi_i \in \partial T$ and $N_i \in \bbN$. We can use the same parameter to define open balls in $\overline{T}$, i.e.\ $\widehat{B_i} = \{ \xi \in \overline{T} \where (\xi, \xi_i)_v \geq N_i\}$. Now
	$$\widehat{U} := \bigcup_{i \in I} \widehat{B_i}$$
	is an open subset of $\overline{T}$ and, clearly, $U \subseteq \widehat{U}$.
	
	Assume $\xi \in \widehat{B_i} \cap \widehat{B_j}$, i.e.\ $(\xi, \xi_i)_v \geq N_i$ and $(\xi, \xi_j)_v \geq N_j$. Without loss of generality we assume $N_j \geq N_i$. Then $(\xi_i, \xi_j)_v \geq \min\{(\xi, \xi_i)_v,(\xi, \xi_j)_v\} \geq N_i$ and hence $\xi_j \in B_i$. Thus, $B_i \cap B_j \neq \emptyset$.
\end{proof}

\begin{lemma}
	\label{lem:enlarge_cover_of_boundary}
	There is a $\pi_1(\Gamma)$-equivariant map
	\begin{align*}
	\{\text{open subsets of } \widetilde V \times \partial \widetilde \Gamma\} &\rightarrow \{\text{open subsets of } \widetilde V \times \overline{\widetilde \Gamma}\} \\
	U &\mapsto \widehat{U}
	\end{align*}
	such that $\;\;U \subseteq \widehat{U}$, $\;\;U \cap V = \emptyset \Rightarrow \widehat{U} \cap \widehat{V} = \emptyset\;\;$ and $\;\;\pr_{\widetilde{V}}(U) = \pr_{\widetilde{V}}(\widehat U)$.
\end{lemma}

\begin{proof}
	Let $U \subseteq \widetilde V \times \partial \widetilde \Gamma$ be open. We can therefore write $U$ as 
	$$U = \bigcup_{i \in I} U_i' \times U_i'',$$
	where $U_i' \subseteq \widetilde V$ and $U_i'' \subseteq \partial \widetilde \Gamma$ are open subsets. Choose a representative for every $\pi_1(\Gamma)$-orbit of open subsets of $\widetilde V \times \partial \widetilde \Gamma$ and set \vspace{-1ex}
	$$\widehat{U} = \bigcup_{i \in I} U_i' \times \widehat{U_i''},$$
	where $\widehat{U_i''}$ is the open subset of $\overline{\widetilde \Gamma}$ with $U_i'' \subseteq \widehat{U_i''}$.
	For all other open subsets we define $\widehat{g.U} := g. \widehat{U}$. Clearly, $U \subseteq \widehat{U}$ and the map $U \mapsto \widehat{U}$ is $\pi_1(\Gamma)$-equivariant.
	
	\vspace{1ex}
	We show the second property via contraposition.
	Assume $\widehat{U} \cap \widehat{V} \neq \emptyset$.
	\begin{align*}
	\emptyset \neq \widehat{U} \cap \widehat{V} &= \bigg( \bigcup_{i \in I} U_i' \times \widehat{U_i''} \bigg) \cap \bigg( \bigcup_{j \in J} V_j' \times \widehat{V_j''} \bigg)  \\
	&= \bigcup_{(i,j) \in I \times J} (U_i' \times \widehat{U_i''}) \cap (V_j' \times \widehat{V_j''}) \\
	&= \bigcup_{(i,j) \in I \times J} (U_i' \cap V_j') \times (\widehat{U_i''} \cap \widehat{V_j''})
	\end{align*}
	$$\Leftrightarrow \exists (i,j) \in I \times J : U_i' \cap V_j' \neq \emptyset \text{ and } \widehat{U_i''} \cap \widehat{V_i''} \neq \emptyset$$
	Now, Lemma \ref{lem:enlarge_open_sets_from_boundary} tells us that $U_i'' \cap V_j'' \neq \emptyset$. 
	Eventually, the same calculation as above, this time for $U$ and $V$ instead of $\widehat{U}$ and $\widehat{V}$, shows $\emptyset \neq \bigcup_{i,j} (U_i' \cap V_j') \times (U_i'' \cap V_j'') =$ $\dots$ $= U \cap V$.
\end{proof}

\begin{lemma}
	\label{lem:boundary_thickening}
	Let $\calU$ be a collection of open sets of $\widetilde{V} \times \overline{\widetilde{\Gamma}}$ such that for all $(v, \xi) \in \widetilde{V} \times \partial \widetilde\Gamma$ there is a $U \in \calU$ such that $B_\alpha(v) \times \{\xi\} \subseteq U$.
	Let $K \subseteq \widetilde V$ be compact, i.e.\ finite.
	Then there exists a number $N \in \bbN$ such that for all 
	$$(v, \xi) \in K \times \set{\xi \in \overline{\widetilde \Gamma}}{\exists \zeta \in \partial \widetilde\Gamma: (\zeta, \xi)_{v_0} \geq N}$$ 
	there is a $U \in \calU$ such that $B_\alpha(v) \times \{\xi\} \subseteq U$.
\end{lemma}

\begin{proof}
	Let $v \in K$ and $\xi \in \partial\widetilde\Gamma$.
	By assumption, there is a $U \in \calU$ such that $B_\alpha(v) \times \{\xi\} \subseteq U$. Since $U$ is open, a small thickening of $B_\alpha(v) \times \{\xi\}$ in $\overline{\widetilde\Gamma}$-direction is still within $U$, i.e.
	$$ \exists N_{v, \xi} \in \bbN: B_\alpha(v) \times B(\xi, N_{v, \xi}) \subseteq U.$$
	For fixed $v$, the sets $B(\xi, N_{v,\xi})$ cover the whole boundary $\partial\widetilde\Gamma$.
	Since the boundary is compact, we can choose a finite subcover by choosing an appropriate finite subset $\Xi_v \subseteq \partial\widetilde\Gamma$.
	Now set
	$$N := \max \set{N_{v,\xi}}{v \in K, \xi \in \Xi_v}.$$
	Assume $v \in K$ and $\xi \in \overline{\widetilde{\Gamma}}$ such that there is a $\zeta \in \partial\widetilde{\Gamma}$ with $(\zeta, \xi)_{v_0} \geq N$.
	There must be a $\xi' \in \Xi_v$ such that $\zeta \in B(\xi', N)$, which automatically implies $\xi \in B(\zeta, N) = B(\xi', N)$.
	Consequently, we can find a $U \in \calU$ with
	\begin{equation*}
	B_\alpha(v) \times \{\xi\} 
	\;\subseteq\; B_\alpha(v) \times B(\xi', N) 
	\;\subseteq\; B_\alpha(v) \times B(\xi', N_{v, \xi'}) 
	\;\subseteq\; U. \qedhere
	\end{equation*}
\end{proof}
	\section{The Lower Left Triangle (Step 3)}

There are several ways to modify an additive category without affecting its $K$-theory. The reason to modify a given additive category is to make room for constructions that could not be made in the original category.

\subsection*{Chain Complexes}

In this section we will closely follow the definitions and constructions of \cite[Section 6]{Bartels.2007} and adjust them to our situation.

\begin{definition}
	Let $G$ be a group and $\calA$ an additive category.
	Define the $\varepsilon$-filtered category $\overline{\calC}_G(X; \calA)$ exactly as in Definition \ref{def:controlled_categories} but without the restriction on the objects that its modules must be distributed locally finitely over $X$.
	For morphisms, replace Condition \ref{c_rows_columns_finite} by the following. Given a morphism $\varphi: (S, \pi, M) \rightarrow (S', \pi', M')$
	\begin{enumerate}
		\item[1'] the matrix entries $\varphi_{s'}^s$ combine to a morphism $\bigoplus_{s \in S} M(S) \rightarrow \bigoplus_{s' \in S'} M'(s')$ in the category $\calA^\kappa$, where $\kappa$ is a suitably chosen cardinal. 
	\end{enumerate}
	Clearly, $\calC_G(...)$ is a subcategory of $\overline \calC_G(...)$. We have analogous versions for $\calO_{G}(...)$ as well.
\end{definition}

\begin{definition}
	Let $\calA$ be an additive category.
	\begin{enumerate}
		\item The category $\Chgeq \calA$ has chain complexes in $\calA$ as objects that are bounded from below and chain maps as morphisms.
		\item The full subcategory $\Chf\calA \subseteq \Chgeq \calA$ consists of all finite chain complexes.
	\end{enumerate}
	If $\calA$ is an $\varepsilon$-filtered category, we call a chain morphism $\varepsilon$-controlled if it is $\varepsilon$-controlled in every degree.
	This makes $\Chf\calA$ an $\varepsilon$-filtered category.
	In $\Ch^\geq\calA$, however, there are chain morphisms that are not $\varepsilon$-controlled for any $\varepsilon > 0$.
	Let $\Idem(\calA)$ be the idempotent completion of an additive category and let $\calA$ be a full additive subcategory of an additive category $\overline{\calA}$.
	\begin{enumerate}
		\setcounter{enumi}{2}
		\item We define $\chhfd(\calA\subseteq \Idem(\overline{\calA}))$ as the full subcategory of $\Chgeq\Idem(\overline{\calA})$, consisting of those objects that are homotopy retracts of objects in $\Chf\calA$.
	\end{enumerate}
\end{definition}

We are interested in the case $\calA = \calO_G(X \mid Y; \calA)$, $\overline{\calA}= \overline{\calO}_G(X \mid Y; \calA)$ and bounded products thereof.
In any case, we abbreviate $\chhfd(\calA\subseteq \Idem(\overline{\calA}))$ to $\chhfd\calA$, since the category $\Idem(\overline{\calA})$ should be apparent in those cases.

When we consider these categories as Waldhausen categories, we regard chain homotopy equivalences as weak equivalences and degree-wise inclusions of direct summands as cofibrations.

\subsection{The Tilde-Construction}
\label{sec:tilde_construction}

\begin{definition}[Cf.\ {\cite[Subsection 8.2]{Bartels.2005}}]
	Let $\calW$ be a Waldhausen category. We define the Waldhausen category $\widetilde \calW$ as the category whose objects are sequences 
	$$C_0 \xrightarrow{c_0} C_1 \xrightarrow{c_1} C_2 \xrightarrow{c_2} \dots$$
	in $\calW$, in which the $c_n$ are simultaneously cofibrations and weak equivalences. A morphism is an equivalence class of commutative diagrams of the form
	\begin{center}
		\begin{tikzcd}
		C_m \arrow[r, "c_m"] \arrow[d, "f_m"] & C_{m+1} \arrow[r, "c_{m+1}"] \arrow[d, "f_{m+1}"] & C_{m+1} \arrow[r, "c_{m+2}"] \arrow[d, "f_{m+2}"] & \dots\phantom{,}  \\
		D_{m+k} \arrow[r, "d_{m+k}"] & D_{m+k+1} \arrow[r, "d_{m+k+1}"] & D_{m+k+2} \arrow[r, "d_{m+k+2}"] & \dots\text{,}
		\end{tikzcd}
	\end{center}
	where $m, k \in \bbN$. We identify such a diagram with
	\begin{center}
		\begin{tikzcd}
		C_{m+1} \arrow[r, "c_{m+1}"] \arrow[d, "f_{m+1}"] & C_{m+2} \arrow[r, "c_{m+2}"] \arrow[d, "f_{m+2}"] & C_{m+3} \arrow[r, "c_{m+3}"] \arrow[d, "f_{m+3}"] & \dots\phantom{.}  \\
		D_{m+k+1} \arrow[r, "d_{m+k+1}"] & D_{m+k+2} \arrow[r, "d_{m+k+2}"] & D_{m+k+3} \arrow[r, "d_{m+k+3}"] & \dots\phantom{.} 
		\end{tikzcd}
	\end{center}
	and also with
	\begin{center}
		\begin{tikzcd}
		C_m \arrow[r, "c_m"] \arrow[d, "d_{m+k} \comp f_m"] & C_{m+1} \arrow[r, "c_{m+1}"] \arrow[d, "d_{m+k+1} \comp f_{m+1}"] & C_{m+1} \arrow[r, "c_{m+2}"] \arrow[d, "d_{m+k+2} \comp f_{m+2}"] & \dots\phantom{.} \\
		D_{m+k+1} \arrow[r, "d_{m+k+1}"] & D_{m+k+2} \arrow[r, "d_{m+k+2}"] & D_{m+k+3} \arrow[r, "d_{m+k+3}"] & \dots.
		\end{tikzcd}
	\end{center}
\end{definition}

\begin{lemma}
	The inclusion of an additive category $\calA$ into $\Chhfd(\calA)$ induces an isomorphism in $K$-theory.
\end{lemma}

\begin{proof}
	Firstly, the inclusion $\calA \subseteq \chhfd(\calA)$ induces an isomorphism by \cite[ Lemma 6.5]{Bartels.2007}.
	Secondly, \cite[ Proposition 8.2]{Bartels.2005} implies that $\chhfd(\calA) \subseteq \Chhfd(\calA)$ induces an isomorphism as well.
	Here, one should note that the conditions $(M)$, $(H)$ and $(Z)$ defined in \cite[Section 8.2]{Bartels.2005} are indeed satisfied for $\chhfd(\calA)$.
\end{proof}

\subsection{The Controlled Transfer}

\begin{definition}
	\label{def:singular_chain_complex}
	For a metric space $X$ we define an object $C_*^\sing(X) \in \Chgeq\overline{\calC}(X; \Ab)$ as follows.
	The object $(S_n, \pi_n, M_n)$ consists of
	\begin{itemize}[label=]
		\item $S_n := S_n(X)$, the set of all singular $n$-simplices in X,
		\item $\pi_n := \bary$, the map sending $\sigma \in S_n(X)$ to its barycentre $\sigma(\bary(\Delta^n)) \in X$ and
		\item $M_n := \bbZ$, the map which assigns the free abelian group $\bbZ$ to every simplex.
	\end{itemize}
	The boundary map $\partial_n :(S_n, \pi_n, M_n) \rightarrow (S_{n-1}, \pi_{n-1}, M_{n-1})$ is given by 
	$$(\partial_n)^\sigma_{\sigma'} = \sum_{\partial_i \sigma = \sigma'} (-1)^i.$$
	Note that for fixed $\sigma'$ there might be infinitely many $\sigma$ with $\partial_i \sigma = \sigma'$.
\end{definition}

\begin{definition}
	Let $\delta > 0$. 
	For a metric space $X$ we define an object $C_*^{\sing,\delta}(X) \in \Chgeq\overline{\calC}(X; \Ab)$ exactly as in Definition \ref{def:singular_chain_complex} with the difference that $S_n$ only consists of those singular $n$-simplices that are of diameter less than $\delta$.
\end{definition}

\begin{lemma}[{Cf.\ \cite[Lemma 6.9]{Bartels.2007}}]
	\label{lem:finite_chain_complex}
	There exists a finite chain complex $D^\delta_*$ in $\Chf\,\calC(\overline \Gamma)$ whose differentials are $\delta$-controlled together with maps 
	$$C_*^{\sing, \delta}(\overline \Gamma) \xrightarrow{i} D_*^\delta \xrightarrow{r} C_*^{\sing, \delta}(\overline \Gamma)$$
	and a chain homotopy $h: r \comp i \simeq \id$ such that $i$, $r$ and $h$ are $6\delta$-controlled morphisms in $\Ch^\geq\overline\calC(\overline \Gamma; \Ab)$.
\end{lemma}

\begin{proof}
	Let $H: \overline \Gamma \times [0,1] \rightarrow \overline \Gamma$ be a homotopy with $H_0 = \id_{\overline \Gamma}$ and $H_t(\overline \Gamma) \subseteq \Gamma$ for all $t > 0$. Since $\overline \Gamma$ is compact, so is $H_t(\overline \Gamma)$ and hence the smallest subcomplex $K_t \subseteq \Gamma$ containing $H_t(\overline \Gamma)$ is compact as well (i.e.\ $K_t$ consists of finitely many simplices).
	
	Consider the map $t \mapsto \sup_{x \in \overline \Gamma} d(x, H_t(x))$. It is continuous and clearly maps $0$ to $0$. Thus, we can find $\varepsilon > 0$ such that $H_\varepsilon$ maps $\delta$-balls into $2\delta$-balls.
	The map $H_\varepsilon$ and the inclusion $\inc: K_\varepsilon \rightarrow \overline \Gamma$ induce well-defined chain maps
	$$C_*^{\sing, \delta}(\overline \Gamma) \xrightarrow{(H_\varepsilon)_*} C_*^{\sing, 2\delta}(K_\varepsilon) \xrightarrow{\inc_*} C_*^{\sing, 2\delta}(\overline \Gamma)$$
	whose composition is homotopic to the inclusion $C_*^{\sing, \delta}(\overline \Gamma) \rightarrow C_*^{\sing, 2\delta}(\overline \Gamma)$ via a homotopy that is $2\delta$-controlled.
	
	After subdividing $K_\varepsilon$, we can assume that $K_\varepsilon = |A|$ is the geometric realization of an abstract simplicial complex all of whose simplices have diameter smaller than $\delta$. Eventually, \cite[Lemma~6.7~(iii)]{Bartels.2007} states that there is a $2\delta$-controlled homotopy equivalence $C_*(A) \rightarrow C_*^{\sing, 2\delta}(K_\varepsilon)$, so that we can set $D_*^\delta := C_*(A)$.
\end{proof}

Just as in Definition \ref{def:equivariant_metric}, choose $G_n$-equivariant maps $q_n: \widetilde{V}_n \rightarrow G_n$ such that the diameter of $q_n^{-1}(1)$ is bounded by $2 \cdot \diam(\Gamma_n)$.
We will fix these maps throughout the remainder of this section.

For $\delta > 0$, we define a chain complex
$$(C^n_*(\delta))_n \in \Chgeq \prod^\bd_{n \in \bbN} \overline{\calC}_{G_n}(\widetilde{V}_n \times \overline{\widetilde{\Gamma}}_n; \Ab)$$
as follows. Choose base points $v_n \in q_n^{-1}(1) \subseteq \widetilde{V}_n$. The $m$-th module $C^n_m(\delta) = (S_m^n, \pi_m^n, M_m^n)_n$ as an object in $\prod_{n \in \bbN}^\bd \overline{\calC}_{G_n}(\widetilde{V}_n \times \overline{\widetilde{\Gamma}}_n; \Ab)$ is given by
\begin{align*}
S_m^n &= \widetilde{V}_n \times S_m^{\sing, \delta}(\overline{\widetilde{\Gamma}}_n, d_{v_n}) \quad \quad\quad\quad (G_n\text{ acting only on } \widetilde{V}_n)\\
\pi_m^n &: (v, \sigma) \mapsto (v, q_n(v).\bary(\sigma)) \\
M_m^n &: (v, \sigma) \mapsto \bbZ.
\end{align*}
Note that $\pi$ is indeed $G_n$-equivariant. The differential $\partial: C_m^n(\delta) \rightarrow C_{m-1}^n(\delta)$ is given by
$$ \partial^{(v, \sigma)}_{(v', \sigma')} = \begin{cases}
\sum\limits_{\partial_i(\sigma) = \sigma'} (-1)^i & v = v' \vspace{1ex} \\
\quad 0 & \text{otherwise.}
\end{cases}$$
Note that differentials only have non-diagonal support in the $\overline{\widetilde{\Gamma}}_n$-direction. Similarly, using the chain complex $D_*^\delta$ from Lemma \ref{lem:finite_chain_complex}, we define a chain complex $(D_*^n(\delta))_n$ over $\widetilde{V}_n \times \overline{\widetilde{\Gamma}}_n$ via 
\begin{align*}
S_m^n &:= \widetilde{V}_n \times S_m^{\delta,n} \quad\quad\quad\quad\quad\quad\quad\quad\; (G_n\text{ acting only on } \widetilde{V}_n) \\
\pi_m^n &: (v, \sigma) \mapsto (v, q_n(v).\bary(\sigma)) \\
M_m^n &: (v, \sigma) \mapsto \bbZ
\end{align*}
in which $S^{\delta,n}_m$ is the set of $m$-simplices in $A^n$, the abstract simplicial complex we used to define $D_*^\delta$ for $\Gamma = \Gamma_n$.

\begin{lemma}[{Cf.\ \cite[Lemma 6.11]{Bartels.2007}}]
	\label{lem:control_on_homotopy_and_retract}
	Let $\delta > 0$ and $C > 1$. The chain complex $C^n_*(\delta)$ is a homotopy retract of the chain complex $D^n_*(\delta)$. The differentials of $C^n_*(\delta)$ and $D^n_*(\delta)$ as well as the maps and homotopies proving that $C^n_*(\delta)$ is a homotopy retract satisfy the following control condition:
	If $((v,\sigma), (v', \sigma'))$ lies in the support of one of these maps, then $v = v'$ and $d_C((v, \bary(\sigma)), (v', \bary(\sigma')) \leq 6C\delta$.
\end{lemma}

\begin{proof}
	Note that $C_*^n(\delta)$ is the unique $G_n$-invariant chain complex whose restriction to $\{v_n\} \times \overline{\widetilde{\Gamma}}_n \subseteq \widetilde{V}_n \times \overline{\widetilde{\Gamma}}_n$  coincides with $C_*^{\sing, \delta}(\overline{\widetilde{\Gamma}}_n, d_{v_n})$ when considered as a chain complex over $\{v_n\} \times \overline{\widetilde{\Gamma}}_n$ via the identification $\{v_n\} \times \overline{\widetilde{\Gamma}}_n \cong \overline{\widetilde{\Gamma}}_n$. Similarly, one can extend all the maps and homotopies from Lemma~\ref{lem:finite_chain_complex} to maps over $\widetilde{V}_n \times \overline{\widetilde{\Gamma}}_n$. The statement about the support of these maps follows immediately from the definitions.
\end{proof}

\begin{definition}[{Cf.\ \cite[Proposition 6.13]{Bartels.2007}}]
	\label{def:transfer}
	Choose a collection of numbers $\delta^\alpha(n)$, as in Lemma \ref{lem:existence_of_monotone_sequences} below.
	Depending on those choices, define
	$$\trans: \frac{\prod^\bd_n}{\bigoplus_n} \calO_{G_n}(\widetilde\Gamma_n; \calA) \rightarrow \Chhfd \bigg( \frac{\prod^\bd_n}{\bigoplus_n} \calO_{G_n}(\widetilde\Gamma_n \mid (V_n \times \overline{\widetilde\Gamma}_n, d_{C(n)}) ; \calA) \bigg),$$
	which does the following:
	
	On an object $A = (S_n, \pi_n, M_n)_n$, the functor is given by 
	$$(A \otimes C_*^n(\delta^1(n))_n \xrightarrow{\id \otimes \inc} (A \otimes C_*^n(\delta^2(n))_n \xrightarrow{\id \otimes \inc} (A \otimes C_*^n(\delta^3(n))_n \xrightarrow{\id \otimes \inc} \dots\text{,} $$
	where $(S_n, \pi_n, M_n)_n \otimes C_m^n(\delta^\alpha(n))$ is the object 
	$$(T_{n, m}, \rho_{n, m}, N_{n, m})_{n, m} \in \chhfd \bigg(\frac{\prod^\bd_n}{\bigoplus_n} \calO_{G_n}(\widetilde\Gamma_n \mid (V_n \times \overline{\widetilde\Gamma}_n, d_{C(n)}) ; \calA)\bigg)$$
	defined via
	\begin{align*}
	T_{n, m} &:= S_n \times_{q_n} \widetilde{V}_n \times S_m^{\sing, \delta(n)}(\overline{\widetilde{\Gamma}}_n, d_{v_n}) \\[-2ex]
	& \text{in which $G_n$ only acts on } S_n \times_{q_n} \widetilde{V}_n := \begin{cases}
	\big\{ (s,v) \where q_n(\pi_{\widetilde\Gamma_n, n}(s)) = q_n(v) \big\} & n \gg d \\
	\emptyset & \text{otherwise}
	\end{cases}\\
	\rho_{n, m} &: (s, v, \sigma) \mapsto 
	(\pi_{\widetilde\Gamma_n, n}(s), \;
	v, \;
	q_n(v).\bary(\sigma), \;
	\pi_{\bbN, n}(s))\\
	N_{n, m} &: (s, v, \sigma) \mapsto M_n(s) \otimes \bbZ 
	\end{align*}
	whose differentials are given by
	$$ \big(\id \otimes \partial(n)\big)^{(s, v, \sigma)}_{(s', v', \sigma')} = \begin{cases}
	\id_{M_n(s)} \otimes \partial(n)^{(v, \sigma)}_{(v', \sigma')} & s = s' \\ 
	0 & \text{otherwise.}
	\end{cases}$$
	A morphism $\varphi: A \rightarrow B$ is mapped to the morphism whose $n^\text{th}$ component is represented by
	\begin{center}
		\begin{tikzcd}
		A \otimes C^n_*(\delta^\alpha(n)) \arrow[r, "\id \otimes \inc"] \arrow[d, "\varphi_n \otimes {l(n)}"] & A \otimes C^n_*(\delta^{\alpha+1}(n)) \arrow[r, "\id \otimes \inc"] \arrow[d, "\varphi_n \otimes {l(n)}"] & A \otimes C^n_*(\delta^{\alpha+2}(n)) \arrow[r, "\id \otimes \inc"] \arrow[d, "\varphi_n \otimes {l(n)}"] & \dots \phantom{.} \\
		B \otimes C^n_*(\delta^{\alpha+1}(n)) \arrow[r, "\id \otimes \inc"] & B \otimes C^n_*(\delta^{\alpha+2}(n)) \arrow[r, "\id \otimes \inc"] & B \otimes C^n_*(\delta^{\alpha+3}(n)) \arrow[r, "\id \otimes \inc"] & \dots.
		\end{tikzcd}
	\end{center}
	Here, $\alpha = \alpha(\varphi) \in \bbN$ is chosen, for which $\varphi$ is $(\alpha - 4 \cdot \diam(\Gamma_n))$-controlled. The morphisms $\varphi_n \otimes l(n)$ are given by
	$$\big(\varphi_n \otimes l(n)\big)^{(s, v, \sigma)}_{(s', v', \sigma')} = \begin{cases}
	(\varphi_n)^s_{s'} \otimes l(n)_{q_n(v)^{-1} q_n(v')} & \bary(\sigma) = \bary(\sigma') \\
	0 & \text{otherwise,}
	\end{cases}$$
	where $l(n)_g : C^{\sing, \delta^\alpha(n)}(\overline{\widetilde{\Gamma}}_n, d_{v_n}) \rightarrow C^{\sing, \delta^{\alpha}(n)}(\overline{\widetilde{\Gamma}}_n, d_{gv_n}) \subseteq C^{\sing, \delta^{\alpha+1}(n)}(\overline{\widetilde{\Gamma}}_n, d_{v_n})$ is induced by left-multiplication of $g$ on $\overline{\widetilde{\Gamma}}_n$.
\end{definition}

\begin{proposition}[{Cf.\ \cite[Proposition 6.13]{Bartels.2007}}]
	The functor $\trans$ defined in Definition \ref{def:transfer} is well-defined.
\end{proposition}

\begin{proof}
	In order to prove well-definedness we have to check several things.
	
	The object $(S_n, \pi_n, M_n)_n \otimes C_m(\delta(n))$ belongs to $\overline \calO_{G_n}(\widetilde\Gamma_n \mid (\widetilde{V}_n \times \overline{\widetilde{\Gamma}}_n, d_{C(n)}; \calA_n)$: The set $S_n \times_{q_n} S_m^{\sing, \delta(n)}(\overline{\widetilde{\Gamma}}_n, d_{v_n})$ is acted upon freely by $G_n$ because $G_n$ acts freely on $S_n$. All maps involved with this action are $G_n$-equivariant.
	
	The differentials $\id \otimes \partial(n)$ satisfy the boundedness condition that is requested for morphisms in the bounded product: This can be seen directly using Lemma \ref{lem:control_on_homotopy_and_retract} and Lemma \ref{lem:existence_of_monotone_sequences}.2. When going through the definitions, one can easily verify that the differentials are $G_n$-equivariant.
	Lemma \ref{lem:control_on_homotopy_and_retract} also shows that $(S_n, \pi_n, M_n)_n \otimes C_*(\delta(n))$ is homotopy equivalent to a finite chain complex over $\calO_{G_n}(\widetilde\Gamma_n \mid (\widetilde{V}_n \times \overline{\widetilde{\Gamma}}_n, d_{C(n)}; \calA_n)$ in such a way that everything satisfies the con\-trol conditions.
	
	Since $\delta^\alpha(n) \leq \delta^{\alpha + 1}(n)$, cf.\ Lemma \ref{lem:existence_of_monotone_sequences}.1, the map $C^n_*(\delta^\alpha(n)) \rightarrow C^n_*(\delta^{\alpha + 1}(n))$ is an inclusion which, degree-wise, is given by inclusions of direct summands. These inclusions are weak equivalences due to Lemma \ref{lem:existence_of_monotone_sequences}.1.
	Thus, $\id \otimes \inc$ is a cofibration and a weak equivalence.
	
	So far, this shows well-definedness on objects. The morphism $\varphi_n \otimes l(n)$ satisfies the boundedness condition more or less because $\varphi_n$ does: If $(\varphi_n)^s_{s'} \otimes l(n)_{q_n(v)^{-1}q_n(v')} \neq 0$ and $\bary(\sigma) = \bary(\sigma')$, then $d_{\widetilde\Gamma_n}(\pi_{\widetilde\Gamma_n, n}(s), \pi'_{\widetilde\Gamma_n, n}(s')) < \alpha - 4 \cdot \diam(\Gamma_n)$. Consequently, we have that $d_{\widetilde{V}_n}(v,v')$ is bounded by
	$$
	\!\underset{\leq 2 \cdot \diam(\Gamma_n)}{\underbrace{d_{\widetilde{V}_n}(v,r_n(\pi_{\widetilde\Gamma_n, n}(s)))}} + 
	\underset{< d (\alpha - 4 \cdot \diam(\Gamma_n))}{\underbrace{d_{\widetilde{V}_n}(r_n(\pi_{\widetilde\Gamma_n, n}(s)), r_n(\pi'_{\widetilde\Gamma_n, n}(s')))}} + 
	\underset{\leq 2 \cdot \diam(\Gamma_n)}{\underbrace{d_{\widetilde{V}_n}(r_n(\pi'_{\widetilde\Gamma_n, n}(s')), v')}}.$$
	Therefore,
	$$d_{C(n)}((v, \bary(\sigma)), (v', \bary(\sigma')) = d_{\widetilde{V}_n}(v,v') < d\alpha + (4-4d) \cdot \diam(\Gamma_n).$$
	The fact that all squares in the diagram commute is obvious. Again, going through the definitions, one can verify that $\varphi_n \otimes l(n)$ is $G_n$-equivariant. Finally, we verify that the left-multiplication of $g = q_n(v)^{-1}q_n(v')$ actually induces a map $C^{\sing, \delta^\alpha(n)}(\overline{\widetilde{\Gamma}}_n, d_{v_n}) \rightarrow C^{\sing, \delta^{\alpha+1}(n)}(\overline{\widetilde{\Gamma}}_n, d_{v_n})$. More precisely, the induced map is
	$$C^{\sing, \delta^\alpha(n)}(\overline{\widetilde{\Gamma}}_n, d_{v_n}) \rightarrow C^{\sing, \delta^\alpha(n)}(\overline{\widetilde{\Gamma}}_n, d_{g.v_n}) \rightarrow C^{\sing, \delta^{\alpha+1}(n)}(\overline{\widetilde{\Gamma}}_n, d_{v_n}).$$
	Here, the first map is indeed induced by left-multiplication with $g$ and the second map is the inclusion that is well-defined because of Lemma \ref{lem:existence_of_monotone_sequences}.3: Choose $w := q_n(v)^{-1}q_n(v').v_n$ and $w' := v_n$. Then
	\begin{align*}
	d_{\widetilde{V}_n}(w, w') 
	&= d_{\widetilde{V}_n}(q_n(v)^{-1}q_n(v').v_n, v_n) \\
	&= d_{\widetilde{V}_n}(q_n(v').v_n, q_n(v).v_n) \\
	&\leq \underset{\leq 2 \cdot \diam(\Gamma_n)}{\underbrace{d_{\widetilde{V}_n}(q_n(v').v_n,v')}} + \!\!\underset{< \alpha - 4\cdot \diam(\Gamma_n)}{\underbrace{d_{\widetilde{V}_n}(v',v)}}\!\! + \underset{\leq 2 \cdot \diam(\Gamma_n)}{\underbrace{d_{\widetilde{V}_n}(v,q_n(v).v_n)}}
	\;<\; \alpha.
	\end{align*}
	This means we can apply Lemma \ref{lem:existence_of_monotone_sequences}.3 to estimate the diameter of simplices w.r.t.\ the metric $d_{v_n}$ instead of $d_{g.v_n}$.
\end{proof}

\begin{lemma}[{Cf.\ \cite[Lemma 6.14]{Bartels.2007}}]
	\label{lem:existence_of_monotone_sequences}
	Let $C(n)$ be a monotone increasing sequence of numbers. There exists a collection of numbers $\delta^\alpha(n) > 0$ with $\alpha, n \in \bbN$ such that the following conditions are satisfied.
	\begin{enumerate}
		\item For every fixed $n \in \bbN$ the sequence $(\delta^\alpha(n))_{\alpha \in \bbN}$ is increasing, i.e.\
		$$ \delta^1(n) \leq \delta^2(n) \leq \delta^3(n) \leq \dots.$$
		\item For every $\alpha \in \bbN$ there exists $\varepsilon(\alpha) \geq 0$ such that
		$\delta^\alpha(n) \leq \frac{\varepsilon(\alpha)}{C(n)}$
		for all $n \in \bbN$.
		\item Consider $w, w' \in \widetilde V_n$, $\xi, \zeta \in \overline {\widetilde\Gamma}_n$ and $\alpha \in \bbN$. If $d_{\widetilde V_n}(w,w') \leq \alpha$ and $d_{w}(\xi, \zeta) \leq \delta^\alpha(n)$, then
		$d_{w'}(\xi, \zeta) \leq \delta^{\alpha + 1}(n)$.
	\end{enumerate}
\end{lemma}

\begin{proof}
	We can almost copy the proof of \cite[Lemma 6.14]{Bartels.2007} one-to-one.
	Note that instead of a metric $d_g$ for a $g \in G$, we here have the metric $d_w$ for a $w \in \widetilde V_n$.
	This is the reason why we always have to carry an error of $2 \cdot \diam(\Gamma_n)$ or $4 \cdot \diam(\Gamma_n)$, cf.\ Definition \ref{def:equivariant_metric} or Lemma \ref{rem:equivariant_metric}.4.
	If we set
	\begin{align*}
	\widetilde R_L(\delta) &:= \sup \left\{d_v(x, y) \,\bigg|
	\begin{tabular}{c}
	$(v, x, y) \in \widetilde{V}_n \times \overline {\widetilde{\Gamma}}_n \times \overline {\widetilde{\Gamma}}_n \text{ with } d_{\widetilde V_n}(v,v_n) \leq L + 2 \cdot \diam(\Gamma_n)$ \\
	$\text{ and } d_{v_n}(x, y) \leq \delta + 4 \cdot \diam(\Gamma_n)$
	\end{tabular} \right\}
	\end{align*}
	and remember the differences explained above, then we can follow the proof of \cite[Lemma 6.14]{Bartels.2007}.
\end{proof}

\begin{lemma}[{Cf.\ \cite[Lemma 6.16]{Bartels.2007}}]
	\label{lem:transfer_commutatve}
	The diagram
	\begin{center}
		\begin{tikzcd}
		\Chhfd \left( \frac{\prod^\bd_n}{\bigoplus_n} \calO_{G_n}(\widetilde\Gamma_n \mid (V_n \times \overline{\widetilde\Gamma}_n, d_{C(n)}) ; \calA) \right) \arrow[dr, "\Pr_*"] \\
		\frac{\prod^\bd_n}{\bigoplus_n} \calO_{G_n}(\widetilde\Gamma_n; \calA) \arrow[u, "\trans"] \arrow[r, "\inc"] & \Chhfd \left( \frac{\prod^\bd_n}{\bigoplus_n} \calO_{G_n}(\widetilde\Gamma_n; \calA)\right)
		\end{tikzcd}
	\end{center}
	is commutative after applying $K$-theory.
	(The induced diagram is the lower left triangle of the diagram on page \pageref{dia:proof_outline}.)
\end{lemma}

\begin{proof}
	We show that $\Pr_* \comp \trans$ and $\inc$ are naturally weak equivalent functors. The transfer is implemented by tensoring with a chain complex in a controlled way. Given a sequence of chain complexes $C_*(n) \in \Ch_\geq(\calA_n)$ with $G_n$-action we can define
	
	$$- \otimes C_*(n): \frac{\prod_n^\bd}{\bigoplus_n} \calO_{G_n}(\widetilde\Gamma_n; \calA_n) \rightarrow \chhfd \bigg( \frac{\prod_n^\bd}{\bigoplus_n} \calO_{G_n}(\widetilde\Gamma_n; \calA_n) \bigg),$$
	which does the following:
	
	On an object $A = (S_n, \pi_n, M_n)_n$, the functor is given by 
	$$A \otimes C_*(n),$$
	where $(S_n, \pi_n, M_n)_n \otimes C_m(n)$ is an object $(T_{n,m}, \rho_{n,m}, N_{n,m})_{n,m} \in \calO_{G_n}(\widetilde\Gamma_n; \calA_n) \times \bbN$ defined via
	\begin{align*}
	T_{n,m} &:= S_n \times C_m(n) \\
	\rho_{n,m} &:= \pi_n \comp \pr_{S_n} \\
	N_{n,m} &: (s, \sigma) \mapsto M_n(s) \otimes \bbZ = M_n(s)
	\end{align*}
	and the differentials are given by
	$$ \big(\id \otimes \partial(n)\big)^{(s, \sigma)}_{(s', \sigma')} = \begin{cases}
	\id_{M_n(s)} \otimes \partial(n)^{\sigma}_{\sigma'} & s = s' \\ 
	0 & \text{otherwise.}
	\end{cases}$$
	A morphism $\varphi: A \rightarrow B$ is mapped to the morphism $(\varphi_n \otimes l(n))_n$ given by
	$$\big(\varphi_n \otimes l(n)\big)^{(s, \sigma)}_{(s', \sigma')} = \begin{cases}
	\varphi^s_{s'} \otimes l(n)_{q_n(v)^{-1} q_n(v')} & \sigma = \sigma', \;\; v := \pi_{\widetilde\Gamma_n, n}(s), v := \pi'_{\widetilde\Gamma_n, n}(s') \\
	0 & \text{otherwise}
	\end{cases}$$
	in which $l(n)_g : C_*(n) \rightarrow C_*(n)$ is given by left-multiplication of $g$.
	
	Now, for $G_n$-chain complexes $C_*(n)$ and $D_*(n)$, (non-equivariant) chain maps $f_*(n) : C_*(n) \rightarrow D_*(n)$, and an object $A \in \frac{\prod_n^\bd}{\bigoplus_n} \calO_{G_n}(\widetilde\Gamma_n; \calA_n)$, we have an induced (non-equivariant) morphism
	$$ \id_A \otimes f_*(n): A \otimes C_*(n) \rightarrow A \otimes D_*(n). $$
	If $f_*(n)$ is $G_n$-equivariant, then so is the induced morphism and this defines a natural transformation
	$$- \otimes f_*(n) : - \otimes C_*(n) \Rightarrow - \otimes D_*(n).$$
	If $f_*(n)$ is a (non-equivariant) chain homotopy equivalence, then $\id_A \otimes f_*(n)$ is a chain homotopy equivalence in $\chhfd \big( \frac{\prod_n^\bd}{\bigoplus_n} \calO_{G_n}(\widetilde\Gamma_n; \calA_n) \big)$: A homotopy inverse for $\id_A \otimes f_*(n)$ is given by $\id_A \otimes g_*(n)$, where $g_*(n)$ is a homotopy inverse for $f_*(n)$.
	
	Regard the following diagram.
	\begin{center}
		\begin{tikzcd}
		A \otimes C_*^{\sing, \delta^1(n)}(\overline{\widetilde{\Gamma}}_n) \arrow[d, "\id \otimes \inc"] \arrow[r, "\id \otimes \inc"] & A \otimes C_*^{\sing, \delta^2(n)}(\overline{\widetilde{\Gamma}}_n) \arrow[d, "\id \otimes \inc"] \arrow[r, "\id \otimes \inc"] & A \otimes C_*^{\sing, \delta^3(n)}(\overline{\widetilde{\Gamma}}_n) \arrow[d, "\id \otimes \inc"] \arrow[r, "\id \otimes \inc"] & \dots \\
		A \otimes C_*^{\sing}(\overline{\widetilde{\Gamma}}_n) \arrow[d, "\id \otimes \pr_*"] \arrow[r, "\id \otimes \inc"] & A \otimes C_*^{\sing}(\overline{\widetilde{\Gamma}}_n) \arrow[d, "\id \otimes \pr_*"] \arrow[r, "\id \otimes \inc"] & A \otimes C_*^{\sing}(\overline{\widetilde{\Gamma}}_n) \arrow[d, "\id \otimes \pr_*"] \arrow[r, "\id \otimes \inc"] & \dots \\
		A \otimes C_*^{\sing}(\pt) \arrow[r, "\id \otimes \inc"] & A \otimes C_*^{\sing}(\pt) \arrow[r, "\id \otimes \inc"] & A \otimes C_*^{\sing}(\pt) \arrow[r, "\id \otimes \inc"] & \dots
		\end{tikzcd}
	\end{center}
	Using the arguments above, each map in the diagram is a chain homotopy equivalence in $\chhfd \big( \frac{\prod_n^\bd}{\bigoplus_n} \calO_{G_n}(\widetilde\Gamma_n; \calA_n) \big)$. Note that $\pr : \overline{\widetilde{\Gamma}}_n \rightarrow \pt$ is a homotopy equivalence because $\overline{\widetilde{\Gamma}}_n$ is contractible.
	
	By going through the definitions one can easily check that the upper row is exactly given by $\Pr_* \comp \trans$ and the lower row is simply $A \xrightarrow{\id} A \xrightarrow{\id} A \xrightarrow{\id} \dots$. In other words, the lower row is given by the inclusion. The above construction yields the natural transformation between $\Pr_* \comp \trans$ and $\inc$, which is a weak equivalence on each object.
\end{proof}
	\section{Trivial $K$-Theory (Step 4)}

\subsection{Excision}

\begin{definition}
	\label{def:excisive_square}
	Let $\calC$ be an additive category with full subcategories $\calA, \calB, \calV \subseteq \calC$ such that $\calV \subseteq \calA$ and $\calV \subseteq \calB$. We call the diagram of inclusions
	\begin{center}
		\begin{tikzcd}
		\calV \arrow[r] \arrow[d] & \calA \arrow[d] \\
		\calB \arrow[r] & \calC
		\end{tikzcd}
	\end{center}
	an \emph{excisive square} if
	\begin{enumerate}
		\item for all $X \in \calC$ there are $A \in \calA$, $B \in \calB$ such that $X$ is a direct sum of $A$ and $B$, and
		\item if a morphism $f \in \calA$ factors through an object of $\calB$, then $f$ factors through an object of $\calV$.
	\end{enumerate}
\end{definition}

\begin{lemma}
	\label{lem:excisive_square_induces_isomorphism}
	Given an excisive square
	\begin{center}
		\begin{tikzcd}
		\calV \arrow[r] \arrow[d] & \calA \arrow[d] \\
		\calB \arrow[r] & \calC
		\end{tikzcd}
	\end{center}
	the canonical functor $\bigslant{\calB}{\calV} \rightarrow \bigslant{\calC}{\calA}$ is an equivalence of categories.
\end{lemma}

\begin{proof}
	The functor $\bigslant{\calB}{\calV} \rightarrow \bigslant{\calC}{\calA}$ is full because the inclusions are full.
	It is faithful because of the second property of excisive squares.
	It is essentially surjective because of the first property of excisive squares.
\end{proof}

\begin{lemma}
	\label{lem:controlled_categories_yield_excisive_squares}
	Let $X$ and $Y$ be metric spaces with an isometric $G$-action and $G$-invariant subspaces $A, B \subseteq Y$.
	Let $R > 0$ such that
	$$d(A \smallsetminus B,B \smallsetminus A) > R. \vspace{-1ex}$$
	Then \vspace{-1ex}
	\begin{center}
		\begin{tikzcd}[row sep=1.5em, column sep=1.5em]
		\calO_G(X \mid A \cap B; \calA) \arrow[r] \arrow[d]  & \calO_G(X \mid A; \calA) \arrow[d] \\
		\calO_G(X \mid B; \calA) \arrow[r] & \calO_G(X \mid A \cup B; \calA)
		\end{tikzcd}
	\end{center}
	satisfies the following properties:
	\begin{enumerate}
		\item Every object in $\calO_G(X \mid A \cup B; \calA)$ is the direct sum of objects in $\calO_G(X \mid A; \calA)$ and $\calO_G(X \mid B; \calA)$ whose inclusions and projections are 0-controlled.
		\item Given $0 < \varepsilon < R$, if a morphism $\varphi$ of $\calO_{G}(X \mid A \cup B \calA)$ that is $\varepsilon$-controlled in $Y$-direction factors through an object of $\calO_G(X \mid B; \calA)$ via morphisms that are $\varepsilon'$-controlled in $Y$-direction, where $0 < \varepsilon' < R$, then $f$ factors through an object of $\calO_G(X \mid A \cap B; \calA)$ via morphisms that are $\varepsilon'$-controlled in $Y$-direction.
	\end{enumerate}
\end{lemma}

\begin{proof}
	\emph{Ad 1.} Set $S_A := \pi_{A \cup B}^{-1}(A)$ and $S_B := S \smallsetminus S_A$.
	Then we have $(S_A, \pi|_{S_A}, M|_{S_A}) \in \calO_G(X \mid A; \calA)$ and $(S_B, \pi|_{S_B}, M|_{S_B}) \in \calO_G(X \mid B; \calA)$.
	Clearly, $(S, \pi, M)$ is a direct sum of both of these objects whose inclusion and projection morphisms are 0-controlled.
	
	\vspace{1ex}
	\emph{Ad 2.} Let $0 < \varepsilon < R$ and $\varphi : (S, \pi, M) \rightarrow (S', \pi', M')$ be an $\varepsilon$-controlled morphism of $\calO_{G}(X \mid A; \calA)$. (In this proof we always implicitly mean \emph{controlled in $Y$-direction} whenever we say \emph{controlled}.)
	Set $A_0 := A \smallsetminus B$ and $A_1 := A \cap B$.
	Now $\varphi$ is the sum of $\varphi_{00}$, $\varphi_{01}$, $\varphi_{10}$ and $\varphi_{11}$, where
	$$(\varphi_{ij})^s_{s'} = \begin{cases}
	\varphi^s_{s'} & \pi_{A}(s) \in A_i \text{ and } \pi'_{A}(s') \in A_j \\
	0 & \text{otherwise.}
	\end{cases}$$
	Since $\varphi_{01}$, $\varphi_{10}$ and $\varphi_{11}$ factor through their domain and target, they factor through an object in $\calO_G(X \mid A \cap B; \calA)$ via $\varepsilon$-controlled morphisms.
	Assume that in $\calO_{G}(X \mid A \cup B; \calA)$ there is a factorization
	\begin{center}
		\begin{tikzcd}[column sep=0em, row sep=1em]
		(S, \pi, M) \arrow[rr, "\varphi"] \arrow[dr, swap, "\psi"] && (S', \pi', M') \\
		& (T, \rho, N) \arrow[ur, swap, "\chi"]
		\end{tikzcd} $\quad (T, \rho, N) \in \calO_G(X \mid B; \calA)$, 
	\end{center}
	where $\psi$ and $\chi$ are $\varepsilon$-controlled.
	We get a similar factorization for $\varphi_{00}$, namely
	\begin{center}
		\begin{tikzcd}[column sep=0em, row sep=1em]
		(S, \pi, M) \arrow[rr, "\varphi_{00}"] \arrow[dr, swap, "\psi_0"] && (S', \pi', M') \\
		& (T, \rho, N) \arrow[ur, swap, "\chi_0"]
		\end{tikzcd} $\quad (T, \rho, N) \in \calO_G(X \mid B; \calA)$,
	\end{center}
	where 
	$$(\psi_0)^s_{t} = \begin{cases}
	\psi^s_{t} & \pi_{A}(s) \in A_0 \\
	0 & \text{otherwise}
	\end{cases} \quad \text{ and } \quad (\chi_0)^t_{s'} = \begin{cases}
	\chi^t_{s'} & \pi'_{A}(s') \in A_0 \\
	0 & \text{otherwise.}
	\end{cases} $$
	Since $\psi_0$ and $\chi_0$ are $\varepsilon'$-controlled, we have
	$$ (\psi_0)^s_{t} \neq 0 \Rightarrow d(\underset{\in A \smallsetminus B}{\underbrace{\pi_{A}(s)}}, \rho_{B}(t)) < \varepsilon' < R \quad \text{ and } \quad (\chi_0)^t_{s'} \neq 0 \Rightarrow d(\rho_{B}(t), \underset{\in A \smallsetminus B}{\underbrace{\pi'_{X}(s')}}) < \varepsilon' < R$$
	and, since $d(A \smallsetminus B, B \smallsetminus A) > R$, we conclude $\rho_{B}(t) \in A$ and thus $\rho_{B}(t) \in A \cap B$.
	Set 
	$$T' := \{t \in T \where \exists s \in S: (\psi_0)^s_{t} \neq 0 \vee \exists s' \in S': (\chi_0)^t_{s'} \neq 0\}.$$
	Then $\varphi_{00}$ factors through $(T', \rho|_{T'}, N|_{T'})$, which is an object in $\calO_G(X \mid A \cap B; \calA)$.
	Finally, with $\varphi_{00}$, $\varphi_{01}$, $\varphi_{10}$ and $\varphi_{11}$ factoring through objects in $\calO_G(X \mid A \cap B; \calA)$, their sum $\varphi$ also factors through an object in $\calO_G(X \mid A \cap B; \calA)$.
\end{proof}

\begin{lemma}
	\label{lem:important_excisive_square}
	Let $X_n$ and $Y_n$ be metric spaces with isometric $G_n$-action.
	Let $A_n, B_n \subseteq Y_n$ be $G_n$-subspaces with $d(A_n \smallsetminus B_n, B_n \smallsetminus A_n) > R$ for a uniform $R > 0$.
	Then
	\begin{center}
		\begin{tikzcd}[row sep=1.5em, column sep=1.5em]
		\frac{\contr_n}{\bigoplus_n} \calO_{G_n}(X_n \mid A_n \cap B_n; \calA) \arrow[r] \arrow[d]  & \frac{\contr_n}{\bigoplus_n} \calO_{G_n}(X_n \mid A_n; \calA) \arrow[d] \\
		\frac{\contr_n}{\bigoplus_n} \calO_{G_n}(X_n \mid B_n; \calA) \arrow[r] & \frac{\contr_n}{\bigoplus_n} \calO_{G_n}(X_n \mid A_n \cup B_n; \calA)
		\end{tikzcd}
	\end{center}
	is an excisive square.
\end{lemma}

\begin{proof}
	The first property of excisive squares follows immediately from the first property of Lemma \ref{lem:controlled_categories_yield_excisive_squares}.
	We have to take a closer look at the second property of excisive squares.
	
	Let $(\varphi_n)_n: (S_n, \pi_n, M_n)_n \rightarrow (S'_n, \pi'_n, M'_n)_n$ be a morphism of $\frac{\contr_n}{\bigoplus_n} \calO_{G_n}(X_n \mid A_n; \calA)$ factoring through an object of $\frac{\contr_n}{\bigoplus_n} \calO_{G_n}(X_n \mid B_n; \calA)$ as follows
	\begin{center}
		\begin{tikzcd}[column sep=0.5em, row sep=1.5em]
		(S_n, \pi_n, M_n)_n \arrow[rr, "(\varphi_n)_n"] \arrow[dr, "(\psi_n)_n"] && (S'_n, \pi'_n, M'_n)_n. \\
		& (T_n, \rho_n, N_n)_n \arrow[ur, "(\chi_n)_n"]
		\end{tikzcd}
	\end{center}
	Since we quotient out the direct sum, we can assume that $\varphi_n$, $\psi_n$ and $\chi_n$ are trivial for small $n$.
	Then the controlled product assures that, for larger $n$, $\varphi_n$, $\psi_n$ and $\chi_n$ are $\varepsilon$-controlled with $\varepsilon < R$.
	From Lemma \ref{lem:controlled_categories_yield_excisive_squares} we obtain objects $(T'_n, \rho'_n, N'_n) \in \calO_{G_n}(X_n \mid A_n \cap B_n; \calA_n)$ together with factorizations
	\begin{center}
		\begin{tikzcd}[column sep=0.5em, row sep=1.5em]
		(S_n, \pi_n, M_n) \arrow[rr, "\varphi_n"] \arrow[dr, "\psi'_n"] && (S'_n, \pi'_n, M'_n) \\
		& (T'_n, \rho'_n, N'_n) \arrow[ur, "\chi'_n"]
		\end{tikzcd}
	\end{center}
	in which $\psi'_n$ and $\chi'_n$ are as well controlled as $\psi$ and $\chi$ when only considering the control in $Y$-direction.
	In $X_n$-direction, however, there is no assurance that $\psi'_n$ and $\chi'_n$ are $\varepsilon$-controlled for a uniform $\varepsilon > 0$.
	This is where Lemma \ref{lem:gaining_control}, see below, comes into play.
	This lemma allows us to modify the factorization in such a way that the control-conditions for both $X_n$- and $Y_n$-direction are satisfied.
\end{proof}

\begin{lemma}
	\label{lem:gaining_control}
	Let $\varphi$ be a morphism of $\calO_G(X \mid A; \calA)$ which is $\varepsilon$-controlled in $X$-direction.
	Assume there is a factorization
	\begin{center}
		\begin{tikzcd}[column sep=0.5em, row sep=1.5em]
		(S, \pi, M) \arrow[rr, "\varphi"] \arrow[dr, "\psi"] && (S', \pi', M')\\
		& (T, \rho, N) \arrow[ur, "\chi"]
		\end{tikzcd}
	\end{center}
	in which $(T, \rho, N)$ is an object in $\calO_G(X \mid B; \calB)$ and in which $\psi$ and $\chi$ are $\delta$-controlled in $Y$ direction.
	Then there exists a factorization
	\begin{center}
		\begin{tikzcd}
		(S, \pi, M) \arrow[r, "\varphi"] \arrow[d, "\psi'"] & (S', \pi', M')\\
		(T', \rho', N')\arrow[r, "\varphi'"] & (T'', \rho'', N'') \arrow[u, "\chi'"]
		\end{tikzcd}
	\end{center}
	in which $(T', \rho', N')$ and $(T'', \rho'', N'')$ are again objects of $\calO_G(X \mid B; \calB)$ and in which $\psi'$ and $\chi'$ are $\delta$-controlled in $Y$ direction and 0-controlled in $X$-direction.
	Additionally, $\varphi'$ is $\varepsilon$-controlled in $X$-direction and 0-controlled in $Y$-direction.
\end{lemma}

\begin{proof}
	Define 
	\begin{align*}
	T' &:= \set{(s,t) \in S \times T}{\psi^s_t \neq 0} &
	T'' &:= \set{(t,s') \in T \times S'}{\chi^t_{s'} \neq 0} \\
	\rho'&: (s,t) \mapsto (\pi_X(s), \pi_A(t), \pi_\bbN(s)) \quad\;\text{ and }&
	\rho''&: (t,s') \mapsto (\pi'_X(s'), \pi_A(t), \pi'_\bbN(s')) \\
	N'&: (s,t) \mapsto N(t) &
	N''&: (t,s') \mapsto N(t).
	\end{align*}
	This way, we can define
	$$(\psi')^{s}_{(s'',t)} := \begin{cases}
	\psi^s_t & s=s'' \\
	0 & \text{otherwise}
	\end{cases}\quad \text{ and } \quad (\chi')^{(t, s'')}_{s'} := \begin{cases}
	\chi^t_{s'} & s'=s'' \\
	0 & \text{otherwise}
	\end{cases}$$
	as well as
	$$(\varphi')^{(s,t)}_{(t',s')} := \begin{cases}
	\id_{N(t)} & t = t' \wedge \varphi^s_{s'} \neq 0\\
	0 & \text{otherwise.}
	\end{cases}$$
	These morphisms clearly satisfy the control conditions as claimed.
	A direct computation shows $\varphi = \chi' \comp \varphi' \comp \psi'$.
\end{proof}

\begin{corollary}
	\label{cor:mayer_vietoris_sequence}
	Let $X_n$ and $Y_n$ be metric spaces with isometric $G_n$-action.
	Let $A_n, B_n \subseteq Y_n$ be $G_n$-subspaces with $d(A_n \smallsetminus B_n, B_n \smallsetminus A_n) > R$ for a uniform $R > 0$.
	Abbreviate $\calH_*(-) := K_*\left(\frac{\contr_n}{\bigoplus_n} \calO_{G_n}(X_n \mid - ; \calA)\right)$.
	Then there is a Meyer-Vietoris sequence
	$${\dots} \rightarrow \calH_{*}(A_n \cap B_n) \rightarrow \calH_{*}(A_n) \oplus \calH_{*}(B_n) \rightarrow \calH_{*}(A_n \cup B_n) \xrightarrow{\delta} \calH_{*-1}(A_n \cap B_n) \rightarrow {\dots}$$
\end{corollary}

\begin{proof}
	The inclusions $A_n \cap B_n \subseteq A_n$ and $B_n \subseteq A_n \cup B_n$ induce Karoubi filtrations, see \cite{Cardenas.1995}.
	Therefore, the excisive square induces a long exact ladder in $K$-theory
	\begin{center}
		\begin{tikzcd}[row sep=2em, column sep=1em]
		\calH_{*}(A_n \cap B_n) \arrow[r] \arrow[d]  & \calH_{*}(A_n) \arrow[d] \arrow[r] & K_*\left(\frac{\frac{\contr_n}{\bigoplus_n} \calO_{G_n}(X_n \mid A_n; \calA)}{\frac{\contr_n}{\bigoplus_n} \calO_{G_n}(X_n \mid A_n \cap B_n; \calA)}\right) \arrow[d, pos=0.3, "\cong"] \arrow[r, "\partial"] & \calH_{*-1}(A_n \cap B_n) \arrow[d] \\
		\calH_{*}(B_n) \arrow[r] & \calH_{*}(A_n \cup B_n) \arrow[urr, dashed, pos=0.2, out=40, in=230, "\delta"] \arrow[r] & K_*\left(\frac{\frac{\contr_n}{\bigoplus_n} \calO_{G_n}(X_n \mid A_n \cup B_n; \calA)}{\frac{\contr_n}{\bigoplus_n} \calO_{G_n}(X_n \mid B_n; \calA)}\right) \arrow[r, "\partial"] & \calH_{*-1}(B_n)
		\end{tikzcd}
	\end{center}
	in which we find the isomorphism of Lemma \ref{lem:excisive_square_induces_isomorphism}.
	The rest is a standard argument:
	The connective homomorphism $\delta$ is given by the composition in which we use the inverse of the excision isomorphism.
\end{proof}

\subsection{Homotopy Invariance}

The following lemma is a typical statement about homotopy invariance in controlled algebra.
Such statements are used to be proven using Eilenberg swindles, see for instance \cite[Proposition 5.6]{Bartels.2004}.

\begin{lemma}
Let $X_n$ and $Y_n$ be metric spaces with isometric $G_n$-action.
The inclusion $\{0\} \rightarrow [0,1]$ induces an isomorphism
$$ K_*\left(\frac{\contr_n}{\bigoplus_n} \calO_{G_n}(X_n \mid Y_n \times \{0\} ; \calA)\right) \rightarrow K_*\left(\frac{\contr_n}{\bigoplus_n} \calO_{G_n}(X_n \mid Y_n \times [0,1] ; \calA)\right). $$
\end{lemma}

\begin{proof}[Sketch proof]
	The induced morphism is part of a long exact Karoubi sequence whose third term vanishes.
	The third term is given by
	$$ K_*\left(\bigslant{\frac{\contr_n}{\bigoplus_n} \calO_{G_n}(X_n \mid Y_n \times [0,1] ; \calA)}{\frac{\contr_n}{\bigoplus_n} \calO_{G_n}(X_n \mid Y_n \times \{0\} ; \calA)}\right) $$
	and vanishes because of an Eilenberg swindle.
	In order to define such a swindle, we first define the endofunctor $F$ via
	$$F: (S_n, \pi_n, M_n)_n \mapsto (S_n, \widetilde\pi_n, M_n)_n$$
	$$\widetilde\pi_n(s) := \big(\pi_X(s), \big(\pi_Y(s), (1-\nicefrac1n) \cdot \pi_{[0,1]}(s)\big), \pi_\bbN(s)\big)$$
	which is the identity on morphisms.
	This functor pushes all modules towards 0 along the the $[0,1]$-coordinate.
	Since this happens slower for larger $n$, this is well-defined on the controlled product.
	If we take the powers of this functor, then we can add them all up without running into problems with local finiteness:
	The only compact sets in which infinitely many modules can accumulate are neighbourhoods of points in $X \times Y \times \{0\} \times \bbN$.
	However, we can ignore those because we quotient them out.
	
	In conclusion, $\sum_{k \in \bbN} F^k$ is a well-defined functor, satisfying
	\begin{equation*}
		\id + \sum_{k \in \bbN} F^k \simeq \sum_{k \in \bbN} F^k. \qedhere
	\end{equation*}
\end{proof}

\begin{lemma}
Let $X_n$, $Y_n$ and $Z_n$ be metric spaces with isometric $G_n$-action with $G_n$-equivariant uniformly continuous maps $f_n, g_n : Y_n \rightarrow Z_n$.
If $f_n$ and $g_n$ are homotopic via a uniformly continuous homotopy, then they induce the same morphism 
$$(f_n)_* = (g_n)_*: K_*\left(\frac{\contr_n}{\bigoplus_n} \calO_{G_n}(X_n \mid Y_n ; \calA)\right) \rightarrow K_*\left(\frac{\contr_n}{\bigoplus_n} \calO_{G_n}(X_n \mid Z_n ; \calA)\right). $$
\end{lemma}

\begin{proof}
	Abbreviate $\calH_*(-) := K_*\left(\frac{\contr_n}{\bigoplus_n} \calO_{G_n}(X_n \mid - ; \calA)\right)$.
	Consider the following diagram.
	\begin{center}
		\begin{tikzcd}[column sep=0.5em, row sep=1.5em]
			\calH_{*}(Y_n) \arrow[rr, "{(f_n)_* ,\; (g_n)_*}"] \arrow[dr, out=-70, in = 175, swap, "{(\inc_0)_* ,\; (\inc_1)_*}"] && \calH_{*}(Z_n) \\
			& \calH_{*}(Y_n \times [0,1]) \arrow[ur, "(H_n)_*"] \arrow{ul}{\cong}[swap]{\pr_*}
		\end{tikzcd}
	\end{center}
	Both $Y_n \rightarrow Y_n \times \{0\} \rightarrow Y_n \times [0,1]$ and $Y_n \rightarrow Y_n \times \{1\} \rightarrow Y_n \times [0,1]$ have the same left-inverse.
	Thus, the same is true for their induced morphisms.
	Since those are isomorphisms by the previous lemma, they must equal.
\end{proof}

\subsection{Induction over the Skeleta}

\begin{lemma}
	\label{lem:trivial_k_theory}
	Let $X_n$ and $Y_n$ be $G_n$-spaces such that $Y_n$ is the geometric realization of uniformly finite-dimensional simplicial complexes and that for each $n \in \bbN$, $G_n$ acts freely on $Y_n$.
	Additionally assume that
	$$K_*\left(\frac{\prod^\bd_n}{\bigoplus_n}\calO(X_n; \calA)\right) = 0.$$	
	Then
	$$K_*\bigg(\frac{\contr_n}{\bigoplus_n} \calO_{G_n}(X_n \mid Y_n ; \calA)\bigg) = 0.$$
\end{lemma}

\begin{proof}
	\emph{Base case.} Assume $Y_n$ is 0-dimensional.
	In that case, $Y_n$ is 1-separated and hence every $\varepsilon$-controlled morphism is already 0-controlled if $\varepsilon < 1$.
	This is expressed in the equality
	$$K_*\bigg(\frac{\contr_n}{\bigoplus_n} \calO_{G_n}(X_n \mid Y_n ; \calA)\bigg) =	K_*\bigg(\frac{\prod^\bd_n}{\bigoplus_n} \calO_{G_n,0}(X_n \mid Y_n;\calA)\bigg)$$
	where $\calO_{G_n,0}(X_n \mid Y_n;\calA)$ is the subcategory of $\calO_{G_n}(X_n \mid Y_n;\calA)$ in which all morphisms are 0-controlled in $Y_n$-direction.
	When we consider $Y_n$ as the union of its $G_n$-orbits we further have
	\begin{align*}
	K_*\bigg(\frac{\prod^\bd_n}{\bigoplus_n} \calO_{G_n,0}(X_n \mid Y_n;\calA)\bigg) &= K_*\bigg(\frac{\prod^\bd_n}{\bigoplus_n} \bigoplus_\text{orbits} \calO_{G_n}(X_n \mid \text{orbit} ; \calA)\bigg) \\
	&= \bigoplus_\text{orbits} K_*\bigg(\frac{\prod^\bd_n}{\bigoplus_n} \calO_{G_n}(X_n \mid \text{orbit} ; \calA)\bigg).
	\end{align*}
	We can also write $\calO_{G_n}(X_n \mid \text{orbit} ; \calA)$ as $\calO_{G_n}(X_n; \calC_0(\text{orbit} ; \calA))$, where $\calC_0(- ; \calA)$ is the subcategory of $\calC(- ; \calA)$ in which all morphisms are 0-controlled.
	Using \cite{Zeggel.2021} we then have an equivalence of $\varepsilon$-filtered categories $\calO_{G_n}(X_n; \calC_0(\text{orbit}; \calA)) \simeq \calO(X_n; \calA)$.
	Putting all this together we obtain
	$$K_*\bigg(\frac{\contr_n}{\bigoplus_n} \calO_{G_n}(X_n \mid Y_n ; \calA)\bigg) = \bigoplus_\text{orbits} K_*\bigg(\frac{\prod^\bd_n}{\bigoplus_n} \calO(X_n; \calA)\bigg) = 0.$$
	
	\vspace{1ex}
	\emph{Induction step.} Assume that $Y_n$ is $(N+1)$-dimensional.
	Choose open neighbourhoods $A_n, B_n \subseteq Y_n$ such that
	\begin{itemize}
		\item $A_n$ deformation retracts to the $N$-skeleton of $Y_n$,
		\item $B_n$ deformation retracts to the union of all barycentres of the $(N+1)$-simplices,
		\item $A_n \cap B_n$ is homotopy equivalent to a union of $N$-spheres, and
		\item $d(A_n \smallsetminus B_n, B_n \smallsetminus A_n) > R$ for some $R > 0$.
	\end{itemize}
	This way we get an according Mayer-Vietoris sequence from Corollary \ref{cor:mayer_vietoris_sequence}.
	For $A_n$, $B_n$ and $A_n \cap B_n$ now either the induction hypothesis or the base case applies. This shows that all terms in the Mayer-Vietoris sequence must vanish.
\end{proof}

\begin{lemma}
	\label{lem:proof_of_assumption}
	The assumption of the Lemma \ref{lem:trivial_k_theory} can be proven for trees using $\nicefrac12$-maps, i.e.\ maps $f: X \rightarrow X$ with $d(f(x), f(y)) \leq \frac12 d(x,y)$ that are Lipschitz homotopic to the identity, cf.\ \cite[Definition 11.1]{Higson.1997}.
\end{lemma}
	
\begin{proof}
	Let $\calO_\varepsilon(X \mid Y; \calA)$ be the subcategory of $\calO(X \mid Y; \calA)$ that has the same objects, but only those morphisms that satisfy the convergence condition ((\ref{eq:convergence}) on page \pageref{eq:convergence}) for both $X$- and $Y$-direction.
	This way we have $\calO(X; \calA) = \calO_\varepsilon(\pt \mid X; \calA)$.
	Now 
	$$\!\!\!\frac{\prod_n^\bd}{\bigoplus_n} \calO(X_n; \calA) 
	= \frac{\prod_n^\bd}{\bigoplus_n} \calO_\varepsilon(\pt \mid X_n; \calA) 
	= \frac{\calO^\lf_\varepsilon(\bigsqcup\nolimits_n \pt \mid X_n; \calA)}{\calO_\varepsilon(\bigsqcup\nolimits_n \pt \mid X_n; \calA)}  
	=: \calO_\varepsilon^\gg\left(\bigsqcup\nolimits_n \pt \mid X_n; \calA\right).$$

	Let $f_n: X_n \rightarrow X_n$ be a sequence of $\nicefrac12$-maps.
	Set $\calB = \calO_\varepsilon^\gg\left(\bigsqcup\nolimits_n \pt \mid X_n; \calA\right)$ and define $f: \calB \rightarrow \calB$ as the product of $(f_n)_*$.
	Note that the control of a morphism, that $f$ is applied to, gets halved.
	By taking higher and higher powers of $f$, we can make the control arbitrarily small, which leads to a well-defined functor $F: \calB \rightarrow \contr_m \calB$ given by $F = \prod_m f^m$.
	
	Consider the following diagram for $k \in \bbN$.
	\begin{center}
		\begin{tikzcd}
			K_*\left(\contr_m \calB\right) \arrow[dr, "(\pr_k)_*"] & K_*\left(\bigoplus_m \calB\right) \arrow[d, "(\pr_k)_*"] \arrow[l, swap, "\cong"] & \bigoplus_m K_*\left(\calB\right) \arrow[l, swap, "\cong"] \arrow[dl, "\pr_k"] \\
			K_*\big(\calB\big) \arrow[r, "\id"] \arrow[u, "F"] & K_*\big(\calB\big)
		\end{tikzcd}
	\end{center}
	If we know that it is commutative and the inclusion $\bigoplus_m \calB \rightarrow \contr_m \calB$ induces an isomorphism, then a simple diagram chase shows that $K_*(\calB) = 0$.
	
	In order to show commutativity we must show $(f^m)_* = \id$.
	This follows from the algebraic version of \cite[Theorem 11.2]{Higson.1997}.
	
	The inclusion $\bigoplus_m \calB \rightarrow \contr_m \calB$ induces an isomorphism because it is part of a long exact Karoubi sequence in which the third term, namely $K_*\big(\frac{\contr_m}{\bigoplus_m} \calB\big)$, vanishes.
	This again is true because of Lemma \ref{lem:trivial_k_theory}, which can be proven for $\calO_\varepsilon^\gg$ in exactly the same way as for $\calO$.
	In this case we have $G = \{1\}$ and the assumption we have to verify is
	$$K_*\left(\frac{\prod_n^\bd}{\bigoplus_n}\calO^\gg\left(\bigsqcup\nolimits_n \pt; \calA\right)\right) = 0.$$
	Since the points in the disjoint union $\bigsqcup_n \pt$ have infinite distance to each other, we can define an Eilenberg-swindle on $\frac{\prod_n^\bd}{\bigoplus_n}\calO^\gg\left(\bigsqcup\nolimits_n \pt; \calA\right)$ by pushing everything towards infinity along the internal $\bbN$-direction.
\end{proof}
	
	\bibliography{bibliography.bib}
	
\end{document}